\documentclass[onepage,hidelinks,onefignum,onetabnum]{siamart251216}
\pdfoutput=1

\usepackage{lipsum}
\usepackage{amsfonts}
\usepackage{graphicx}
\usepackage{epstopdf}
\usepackage{algorithmic}
\usepackage{enumitem}
\usepackage{mathtools}
\usepackage{graphicx}
\usepackage{subcaption}
\usepackage{comment}
\ifpdf
  \DeclareGraphicsExtensions{.eps,.pdf,.png,.jpg}
\else
  \DeclareGraphicsExtensions{.eps}
\fi

\newsiamremark{remark}{Remark}
\newsiamthm{assumption}{Assumption}
\newsiamremark{hypothesis}{Hypothesis}
\crefname{hypothesis}{Hypothesis}{Hypotheses}
\newsiamthm{claim}{Claim}

\headers{Non-Expansive Two-Time-Scale Iterations}{S. Chandak}

\title{Non-Expansive Mappings in Two-Time-Scale Stochastic Approximation: Finite-Time Analysis}

\author{Siddharth Chandak\thanks{Electrical Engineering, Stanford University, CA, USA.
  (\email{chandaks@stanford.edu}).}}

\usepackage{amsopn}

\ifpdf
\hypersetup{
  pdftitle={Non-expansive Two-Time-Scale},
  pdfauthor={S. Chandak}
}
\fi

\newcommand{\xstar}{x^*}
\newcommand{\xhat}{\hat{x}}
\newcommand{\ystar}{y^*}
\newcommand{\yhat}{\hat{y}}
\newcommand{\EE}{\mathbb{E}}
\newcommand{\FF}{\mathcal{F}}
\newcommand{\PP}{\mathcal{P}}
\newcommand{\XX}{\mathcal{X}}
\newcommand{\Ystar}{\mathcal{Y}^*}
\newcommand{\Yhat}{\mathcal{\hat{Y}}}
\newcommand{\RR}{\mathbb{R}}
\newcommand{\cc}{\mathfrak{c}}
\newcommand{\afrak}{\mathfrak{a}}
\newcommand{\bfrak}{\mathfrak{b}}
\newcommand{\argmin}{\arg\min}
\newcommand{\gradJ}{\nabla J}

\makeatletter

\theoremstyle{plain}
\theoremheaderfont{\normalfont\bfseries}
\theorembodyfont{\normalfont\itshape}
\theoremseparator{.}
\theoremsymbol{}

\renewtheorem{theorem}{Theorem}[section]

\renewtheorem{lemma}[theorem]{Lemma}
\renewtheorem{corollary}[theorem]{Corollary}
\renewtheorem{proposition}[theorem]{Proposition}
\renewtheorem{definition}[theorem]{Definition}
\renewtheorem{assumption}[theorem]{Assumption}

\makeatother

\begin{document}

\maketitle

\begin{abstract}
Two-time-scale stochastic approximation algorithms are iterative methods used in applications such as optimization, reinforcement learning, and control. Finite-time analysis of these algorithms has primarily focused on fixed point iterations where both time-scales have contractive mappings. In this work, we broaden the scope of such analyses by considering settings where the slower time-scale has a non-expansive mapping. For such algorithms, the slower time-scale can be viewed as a stochastic inexact Krasnoselskii-Mann iteration. We also study a variant where the faster time-scale has a projection step which leads to non-expansiveness in the slower time-scale. We show that the last-iterate mean square residual error for such algorithms decays at a rate $\mathcal{O}(1/k^{1/4-\epsilon})$, where $\epsilon>0$ is arbitrarily small. We further establish almost sure convergence of iterates to the set of fixed points. We demonstrate the applicability of our framework by applying our results to minimax optimization, linear stochastic approximation, and Lagrangian optimization.
\end{abstract}



\section{Introduction}
Stochastic Approximation (SA) is a class of iterative algorithms to find the fixed point of an operator given its noisy realizations \cite{Robbins-Monro}. These algorithms have found applications in various areas, including optimization, reinforcement learning \cite{Sutton}, communications and stochastic control \cite{Borkar-book}. Two-time-scale stochastic approximation is a variant of SA, where the iteration is as follows:
\begin{equation*}
    \begin{aligned}
    &x_{k+1}=x_k+\alpha_k(f(x_k,y_k)-x_k+M_{k+1})\\
    &y_{k+1}=y_k+\beta_k(g(x_k,y_k)-y_k+M'_{k+1}).
\end{aligned}
\end{equation*}
Here $x_k\in\RR^{d_1}$ and $y_k\in\RR^{d_2}$ are the coupled iterates updated on two separate time-scales, decided by the step sizes $\alpha_k$ and $\beta_k$. $M_{k+1}$ and $M'_{k+1}$ are noise sequences. The iterate $x_k$ updates on the faster time-scale, i.e., the step size $\alpha_k$ decays slowly. While analyzing the behavior of iterate $x_k$, the iterates $y_k$ are considered quasi-static or slowly moving. The iterate $y_k$ updates on the slower time-scale, i.e., the step size $\beta_k$ decays faster. The behavior of iterate $y_k$ is often studied in the limit where $x_k$ closely tracks $\xstar(y_k)$, where $\xstar(y)$ is the fixed point for $f(\cdot,y)$. The step sizes have to be appropriately chosen for the algorithm to converge.

The majority of analysis for two-time-scale algorithms has occurred under the assumption that the function $f(x,y)$ is contractive in $x$, and the function $g(\xstar(y),y)$ is contractive \cite{ Chandak-TTS-Opti,Doan, Konda}. Under these assumptions, the algorithm converges to $\xstar,\ystar$, the unique solution to the equations: $f(x,y)=x$ and $g(x,y)=y$ \cite{Borkar-book}. In recent years, there have been works that obtain finite-time bounds for iterates, but they also mainly focus on contractive mappings in both time-scales \cite{Chandak-TTS-Opti,Doan}.

In this paper, we consider a non-expansive mapping in the slower iteration, i.e., the mapping $g(\xstar(y),y)$ is non-expansive. Such two-time-scale iterations arise in applications including minimax optimization \cite{jordan-ttsgda}, Lagrangian optimization \cite{two-time-lagrangian}, and generalized Nash equilibrium problems \cite{Chandak}. The slower time-scale in these iterations can be considered a stochastic inexact Krasnoselskii-Mann (KM) algorithm (an iterative algorithm to find the fixed points of a non-expansive operator). Analysis of these iterations requires novel techniques to deal with the errors due to two-time-scale updates in a stochastic setting with non-expansive maps. 

Our main contributions are as follows.
\begin{enumerate}
    \item We present a finite-time bound for our algorithm, which has contractive and non-expansive mappings in the faster and slower time-scales, respectively. We show that the last-iterate mean square residual error decays at a rate $\mathcal{O}(k^{-1/4+\epsilon})$, where $\epsilon>0$ can be arbitrarily small depending on the step size.
    \item We show that the iterates converge to the set of fixed points almost surely. 
    \item We provide a similar result for a projected variant, where the iterates $x_k$, updated in the faster time-scale, are projected to a convex set after each iteration. We show that there exist scenarios where the projection in the faster time-scale leads to the non-expansive mapping (Subsection \ref{subsec:lagrangian}).
    \item We adapt strongly concave-convex minimax optimization, linear SA, and Lagrangian optimization into our framework.
    \item Using similar techniques, we present an additional result where the slower time-scale iteration contains a gradient descent operator for a smooth function. The gradient descent operator is not necessarily non-expansive in this case. In this setting, we show that the minimum, from time $0$ to $k$, of the expected norm of the gradient is bounded by $\mathcal{O}(\log(k)/k^{2/5})$.
\end{enumerate}

\subsection{Related Work}
Applications of SA to reinforcement learning (RL) and optimization have caused a growing interest in obtaining results which bound their finite-time performance. These include mean square error or moment bounds (e.g., \cite{Zaiwei, Srikant}) and high probability or concentration bounds (e.g., \cite{Chandak-conc, Zaiwei-conc}). 

Two-time-scale algorithms with contractive mappings (in both time-scales) play an important role in RL \cite{Sutton-GTD}, and hence finite time bounds for such algorithms have also been studied in detail. In this setting, the bound is obtained on $\EE[\|x_k-\xstar(y_k)\|^2+\|y_k-\ystar\|^2]$. A major focus in this field is on linear two-time-scale SA, i.e., where the maps $f(\cdot)$ and $g(\cdot)$ are linear. In \cite{Konda}, an \textit{asymptotic} convergence rate of $\mathcal{O}(1/k)$ is obtained, while in \cite{tts-new-2, Dalal-linear,Shaan,Kaledin-linear} the result is extended to a finite-time bound. We also study linear two-time-scale SA as a special case. But we study the case where the matrix relevant for the slower time-scale is negative semidefinite instead of Hurwitz or negative definite as assumed in the above works.

Bounds for contractive non-linear two-time-scale SA have also been studied in recent years. A rate of $\mathcal{O}(1/k^{2/3})$ was previously established for this setting in \cite{Doan}. This rate has since been improved to $\mathcal{O}(1/k)$ in \cite{Chandak-TTS-Opti}. Rates for two-time-scale SA with contractions under arbitrary norms and Markovian noise were obtained in \cite{Chandak-TTS-arbitrary}.

Two-time-scale SA has also been studied in settings where the mappings are not contractive. In \cite{Chandak}, the slower time-scale has a non-expansive mapping, but the result is specific to control of strongly monotone games. In \cite{Wu-actor-critic,Doan-actor-critic}, the slower time-scale contains a gradient descent operator $\gradJ(y_k)$ for a smooth function $J(\cdot)$. The operator $\gradJ(\cdot)$ is not co-coercive in general and is different from our framework. Under this setting, bounds on $\min_{0\leq i\leq k}\EE[\|\gradJ(y_i)\|^2]$ are obtained. We give an additional result (Theorem \ref{thm:gradient}), where we match their bound of $\mathcal{O}(\log(k)/k^{2/5})$. We work under assumptions different from \cite{Doan-actor-critic, Wu-actor-critic} where they consider Markovian noise but require boundedness of some operators. We also consider the setting where $J(\cdot)$ is convex, which fits into the framework of Theorem \ref{thm:main}, and obtain a bound on $\EE[\|\gradJ(y_k)\|^2]$.

Our slower time-scale can be written as a stochastic inexact KM iteration. Finite-time bounds for inexact KM iterations have been obtained under a variety of settings \cite{Davis-splitting,KM-book,Liang-inexact,inertial-KM}. In recent years, there have also been works which give bounds for stochastic KM iterations in Banach spaces \cite{Bravo-1,Bravo-2}. Our analysis requires an appropriate combination of errors due to two-time-scale updates with the techniques used for stochastic inexact KM iterations. 

\subsection{Outline and Notation}The paper is structured as follows: Section \ref{sec:main} sets up and presents the main result. The projected variant is presented in Subsection \ref{subsec:projected}. Section \ref{sec:outline} contains a proof sketch for our main result. Section \ref{sec:gradient-result} presents the framework where the slower time-scale has a gradient descent operator. Section \ref{sec:applications} presents the above-listed applications and how they can be incorporated into our framework, and Section \ref{sec:experiments} presents some numerical experiments for these applications. Finally, Section \ref{sec:conclusion} concludes the paper and presents some future directions. 

Throughout this work, $\|\cdot\|$ denotes the Euclidean norm, and $\langle x_1,x_2\rangle$ denotes the inner product given by $x_1^Tx_2$.

\section{Non-Expansive Mappings in Slower Time-Scale}\label{sec:main}
The primary focus of this paper is on two-time-scale iterations where the faster and slower time-scales are fixed point iterations with contractive and non-expansive mappings, respectively. We present the main results of this paper in this section after setting up the necessary notation and assumptions. Consider the following coupled iterations. 
\begin{equation}\label{iter-main}
    \begin{aligned}
    &x_{k+1}=x_k+\alpha_k(f(x_k,y_k)-x_k+M_{k+1})\\
    &y_{k+1}=y_k+\beta_k(g(x_k,y_k)-y_k+M'_{k+1}).
\end{aligned}
\end{equation}
Here $x_k\in\RR^{d_1}$ is the iterate updating on the faster time-scale, and $y_k\in\RR^{d_2}$ is the iterate updating on the slower time-scale. $\alpha_k$ and $\beta_k$ are the respective step sizes that dictate the update speed and are formally defined in Assumption \ref{assu:stepsize}. $M_{k+1}\in\RR^{d_1}$ and $M'_{k+1}\in\RR^{d_2}$ are martingale difference noise sequences (Assumption \ref{Martingale}). 

We begin by assuming that the function $f(x,y)$ is contractive with respect to $x$.
\begin{assumption}\label{f-contrac}
The function $f:\RR^{d_1}\times \RR^{d_2}\mapsto \RR^{d_1}$ is contractive in $x$ for each $y\in\RR^{d_2}$. Specifically, there exists a constant $0\leq \mu<1$, independent of $y$, such that
   $$\|f(x_1,y)-f(x_2,y)\|\leq \mu \|x_1-x_2\|,$$
for all $x_1,x_2\in\RR^{d_1}$ and $y\in\RR^{d_2}$.
\end{assumption}
By the Banach contraction mapping theorem, the above assumption implies that $f(\cdot,y)$ has a unique fixed point for each $y$. That is, for each $y\in\RR^{d_2}$, there exists unique $\xstar(y)$ such that $f(\xstar(y),y)=\xstar(y)$. The following is an assumption on the function $g:\RR^{d_1}\times\RR^{d_2}\mapsto \RR^{d_2}$.
\begin{assumption}\label{g-nonexp}
The function $g(\xstar(\cdot),\cdot):\RR^{d_2}\mapsto \RR^{d_2}$ is non-expansive, i.e., 
    $$\|g(\xstar(y_1),y_1)-g(\xstar(y_2),y_2)\|\leq \|y_1-y_2\|,$$
for all $y_1,y_2\in\RR^{d_2}$. Let us denote the set of fixed points of $g(\xstar(\cdot),\cdot)$ by $\Ystar$, i.e., $\Ystar=\{y\in\RR^{d_2}\mid g(\xstar(y),y)=y\}$. We assume that the set $\Ystar$ is non-empty.
\end{assumption}

Before discussing other required assumptions, we present alternate but equivalent formulations of the above assumptions. Define operator $T_1(x,y)=x-f(x,y)$. Then Assumption \ref{f-contrac} is equivalent to the operator $T_1(x,y)$ being strongly monotone in $x$:
$$\langle T_1(x_1,y)-T_1(x_2,y),x_1-x_2\rangle\geq (1-\mu)\|x_1-x_2\|^2.$$
Similarly, define operator $T_2(y)=y-g(\xstar(y),y)$. Then the non-expansiveness in Assumption \ref{g-nonexp} is equivalent to the operator $T_2$ being co-coercive, i.e., 
$$\langle T_2(y_1)-T_2(y_2),y_1-y_2\rangle \geq \frac{1}{2}\|T_2(y_1)-T_2(y_2)\|^2.$$
For completeness, the complete equivalence, including the proofs for the above statements and the converse directions, is presented in Appendix \ref{app:relations}.

Next, we present two assumptions required for our analysis. These are standard assumptions associated with analysis of SA iterations \cite{Borkar-book,Zaiwei}. The first assumption states that the functions $f$ and $g$ are Lipschitz. 
\begin{assumption}\label{Lipschitz}
Functions $f(\cdot,\cdot)$ and $g(\cdot,\cdot)$ are $L$-Lipschitz, i.e., 
    \begin{equation*}
    \|f(x_1,y_1)-f(x_2,y_2)\|+\|g(x_1,y_1)-g(x_2,y_2)\|\leq L(\|x_1-x_2\|+\|y_1-y_2\|).
\end{equation*}
for all $x_1,x_2\in\RR^{d_1}$ and $y_1,y_2\in\RR^{d_2}$. For $x\in\RR^{d_1}$ and $y\in\RR^{d_2}$, we also assume that 
    \begin{equation*}
    \|f(x,y)\|+\|g(x,y)\|\leq \cc_1(1+\|x\|+\|y\|).
\end{equation*}
\end{assumption}
The next assumption is that $M_{k+1}$ and $M'_{k+1}$ are martingale difference sequences.
\begin{assumption}\label{Martingale}
Define the family of $\sigma$-fields $\FF_k=\sigma(x_0,y_0,M_i,M'_i, i\leq k)$. Then $\{M_{k+1}\}$ and $\{M'_{k+1}\}$ are martingale difference sequences with respect to $\FF_k$, i.e., $\EE[M_{k+1}\mid\FF_k]=\EE[M'_{k+1}\mid\FF_k]=0.$
Moreover, 
    $\EE[\|M_{k+1}\|^2+\|M'_{k+1}\|^2\mid\FF_k]\leq \cc_2 (1+\|x_k\|^2+\|y_k\|^2), \;\forall k\geq 0.$
\end{assumption}

We finally make an assumption about the step size sequences, which is necessary to ensure the time-scale separation between the two updates.
\begin{assumption}\label{assu:stepsize}
The step sizes $\alpha_k$ and $\beta_k$ are of the form $$\alpha_k=\frac{\alpha}{(k+K_1)^{\afrak}}\;\;\text{and}\;\;\beta_k=\frac{\beta}{(k+K_1)^{\bfrak}},$$
where $0.5<\afrak<\bfrak<1$, $\alpha,\beta\leq 0.5$ and $K_1\geq 1$. This ensures that $(\alpha_k)$ is non-summable $(\sum_k\alpha_k=\infty)$, square-summable $(\sum_k\alpha^2_k<\infty)$ and non-increasing $(\alpha_{k+1}\leq \alpha_k)$, and analogously for $(\beta_k)$. This also ensures that $\lim_{k\uparrow\infty} \beta_k/\alpha_k=0$. Finally, we make the key assumption that 
$\beta_k^2/\alpha_k^3\leq 1.$
\end{assumption}
We restrict our attention to step sizes of this form for simplicity but our results can be extended to other step size sequences. We do not include $\Theta(1/k)$ in our set of step sizes because it is not optimal for non-expansive fixed point iterations. The assumptions that the step sizes are non-summable, square summable, and satisfy $\beta_k/\alpha_k$ goes to zero as $k\uparrow\infty$ are standard in stochastic approximation literature \cite{Borkar-book}. The assumption of $\beta_k^2/\alpha_k^3\leq 1$ is the assumption that decides the time-scale separation between the updates. The assumption is key for our analysis as it is the deciding factor in the bound on the mean square error. It is required to show that the iterates are bounded in expectation (Lemma \ref{lemma:bounded-expectation}) which is a key step in the proof. Similar assumptions have also been taken in papers such as \cite{Doan}.

Define $\gamma_1\coloneqq(1-\mu)/(2L^2)$ and $\gamma_2\coloneqq\max\left\{\left(2\afrak/(\mu\alpha)\right)^{1/1-\afrak},\left(2\bfrak/\beta\right)^{1/1-\bfrak}\right\}$. The following is our main result. An outline of its proof is given in Section \ref{sec:outline}, and the complete proof has been presented in Appendix \ref{app:main-proof}.
\begin{theorem}\label{thm:main}
    Suppose that Assumptions \ref{f-contrac}-\ref{assu:stepsize} hold. Let $\{x_k,y_k\}$ be generated by \eqref{iter-main} with $\beta/\alpha\leq \gamma_1$ and $K_1\geq \gamma_2$. Then, 
    \begin{itemize}
        \item  There exist constants $C_1,C_2>0$ such that for all $k\geq 0$,
        $$\EE\left[\|x_k-\xstar(y_k)\|^2\right]\leq \frac{C_1}{(k+1)^\afrak} \;\;\text{and}\;\; \EE\left[\|g(\xstar(y_k),y_k)-y_k\|^2\right]\leq \frac{C_2}{(k+1)^{1-\bfrak}}.$$
    \item Moreover, $\|x_k-\xstar(y_k)\|$ converges to zero, and the iterates $y_k$ converge to the set $\Ystar$ with probability $1$.
    \end{itemize}
\end{theorem}

\subsubsection*{Best Rate and Step size Choice}
The best rate for the algorithm based on the above analysis is $1/(k+1)^{0.25-\epsilon}$ where $\epsilon>0$ can be arbitrarily small. This can be achieved by choosing $\afrak=0.5+(2/3)\epsilon$ and $\bfrak=0.75+\epsilon$. Unlike contractive maps, in the analysis of fixed point iterations with non-expansive mappings, a slowly decreasing step size gives a better rate. Hence, a slowly decreasing $\beta_k$, or equivalently a smaller $\bfrak$, gives a better rate. This, along with the assumptions that $\afrak>0.5$ and $\beta_k^2/\alpha_k^3\leq 1$ (equivalently, $2\bfrak\geq 3\afrak$), dictates the optimal rate. 

We also point out that $\EE\!\left[\|x_k-\xstar(y_k)\|^2\right]$ and 
$\EE\!\left[\|g(\xstar(y_k),y_k)-y_k\|^2\right]$ scale as 
$\mathcal{O}(\alpha_k)$ and $\mathcal{O}(1/(k\beta_k))$, respectively. 
For the given stepsize sequences $\{\alpha_k\}$ and $\{\beta_k\}$, these rates are the best known upper bounds for stochastic approximation with contractive and non-expansive operators, respectively, even in the single-time-scale setting \cite{Zaiwei, Bravo-1}. 
Consequently, when optimizing over all stepsize sequences admissible under Assumption~\ref{assu:stepsize}, the overall rate is fundamentally constrained by the stepsize requirements, and could potentially be improved only if these constraints were relaxed.

The assumptions on $\beta/\alpha$ and $K_1$ are required only for the bounds to hold for all $k\geq 0$. Instead, if we have that $\beta_k/\alpha_k\leq \gamma_1$ for all $k\geq K_0$, then the bounds will hold for all $k\geq \max\{K_0,\gamma_2\}$. The existence of such $K_0$ is guaranteed by Assumption \ref{assu:stepsize}. 

\subsubsection*{Comparison with Rates for Contractive SA}
The mean square error rate of $\mathcal{O}(k^{-1/4+\epsilon})$ is weaker than bounds observed for contractive SA. But such weaker bounds are frequently observed in non-expansive variants of SA. The following table compares bounds for contractive and non-expansive SA under different settings.

\begin{table}[htbp]
\scriptsize
\caption{Comparison of current best bounds for SA schemes under contractive and non-expansive mappings. The Euclidean and arbitrary norms refer to the norms under which the mapping satisfies contractive or non-expansive properties, respectively.}
\begin{center}\label{table:comp}
\begin{tabular}{|p{2.7cm}|p{3.7cm}|p{5.3cm}|}
\hline
\textbf{Setting} & \textbf{Contractive Maps} & \textbf{Non-Expansive Maps} \\
\hline
Single Time-Scale SA\newline (Euclidean norm) & \rule{0pt}{1em}$\EE\left[\|x_k-\xstar\|^2\right]=\mathcal{O}\left(1/k\right)$\newline (Stepsize $\mathcal{O}(1/k)$) \cite{Zaiwei} & $\EE\left[\|f(x_k)-x_k\|^2\right]=\mathcal{O}(\log(k)/k^{1/3})$ \newline (Stepsize $\mathcal{O}(1/k^{2/3})$) \\
\hline
Single Time-Scale SA\newline (arbitrary norm) & \rule{0pt}{1em}$\EE\left[\|x_k-\xstar\|^2\right]=\mathcal{O}(1/k)$\newline (Stepsize $\mathcal{O}(1/k)$) \cite{Zaiwei}  & $\EE\left[\|f(x_k)-x_k\|\right]=\mathcal{O}(\log(k)/k^{1/6})$ \newline (Stepsize $\mathcal{O}(1/k^{2/3})$) \cite{Bravo-1}  \\
\hline
Two-Time-Scale SA & \rule{0pt}{1em}$\EE\left[\|x_k-\xstar(y_k)\|^2\right]=\mathcal{O}(1/k)$\newline $\EE\left[\|y_k-\ystar\|^2\right]=\mathcal{O}(1/k)$\newline (Stepsizes $\mathcal{O}(1/k)$) \cite{Chandak-TTS-Opti} & $\EE\left[\|x_k-\xstar(y_k)\|^2\right]=\mathcal{O}(1/k^{0.5})$\newline $\EE\left[\|g(\xstar(y_k),y_k)-y_k\|^2\right]=\mathcal{O}(1/k^{0.25-\epsilon})$\newline \textbf{[Our result]} \\
\hline
\end{tabular}
\end{center}
\end{table}

While a bound of $\mathcal{O}(1/k)$ is achieved in all the cases of contractive SA listed above, the bound achieved by their non-expansive analogs is much weaker. For a stepsize of the form $\mathcal{O}(1/k^\bfrak)$, the rate achieved in non-expansive SA is usually of the order $\mathcal{O}(1/k^{1-\bfrak})$. As we require square-summability, we need $\bfrak>0.5$. The additional assumptions on the stepsize sequences further weaken the bound in our setting.

\subsection{Projected Variant}\label{subsec:projected}
In this subsection, we consider a variant where the iterate in the faster time-scale is projected to a convex and bounded set at each step. 
\begin{equation}\label{iter-projected}
    \begin{aligned}
    &x_{k+1}=\PP_\XX\big(x_k+\alpha_k(f(x_k,y_k)-x_k+M_{k+1})\big)\\
    &y_{k+1}=y_k+\beta_k(g(x_k,y_k)-y_k+M'_{k+1}).
\end{aligned}
\end{equation}
Here $\XX$ is a convex and bounded set, and $\PP_\XX$ denotes the projection onto set $\XX$, i.e., $\PP_\XX(x)=\argmin_{x'\in\XX} \|x-x'\|$. 

We still assume that the function $f(\cdot,y)$ is contractive for all $y$. But the relevant fixed point is now the fixed point for $\PP_\XX(f(\cdot,y))$, i.e., we now define $\xhat(y)$ as the unique point that satisfies $\xhat(y)=\PP_\XX(f(\xhat(y),y))$. In this case, the function $g(\xhat(\cdot),\cdot)$ is assumed to be non-expansive.
\begin{assumption}\label{g-nonexp-proj}
    The function $g(\xhat(\cdot),\cdot):\RR^{d_2}\mapsto \RR^{d_2}$ is non-expansive, i.e., 
    $$\|g(\xhat(y_1),y_1)-g(\xhat(y_2),y_2)\|\leq \|y_1-y_2\|,$$
for all $y_1,y_2\in\RR^{d_2}$. Let us denote the set of fixed points of $g(\xhat(\cdot),\cdot)$ by $\Yhat$, i.e., $\Yhat=\{y\in\RR^{d_2}\mid g(\xhat(y),y)=y\}$. We assume that the set $\Yhat$ is non-empty.
\end{assumption}

This projected variant is studied as a special case because there exist scenarios where the projection of the faster time-scale updates is why the slower time-scale iteration has a non-expansive mapping (instead of a contractive mapping). Formally, there arise scenarios where the map $g(\xstar(\cdot),\cdot)$ is contractive, but the map $g(\xhat(\cdot),\cdot)$ is not guaranteed to be contractive, and is only non-expansive. We study one such scenario in Subsection \ref{subsec:lagrangian}.

The main theorem for this projected variant is similar to Theorem \ref{thm:main}, and its proof is presented in Appendix \ref{app:projected-proof}.
\begin{theorem}\label{thm:main-projected}
    Suppose that Assumptions \ref{f-contrac}, \ref{g-nonexp-proj} and \ref{Lipschitz}-\ref{assu:stepsize} hold. Let $\{x_k,y_k\}$ be generated by \eqref{iter-projected} with $\beta/\alpha\leq \gamma_1$ and $K_1\geq \gamma_2$. Then, 
        \begin{itemize}
        \item   There exist constants $C_1',C_2'>0$ such that for all $k\geq 0$,
        $$\EE\left[\|x_k-\xhat(y_k)\|^2\right]\leq \frac{C_1'}{(k+1)^\afrak},\;\;\text{and}\;\;\EE\left[\|g(\xhat(y_k),y_k)-y_k\|^2\right]\leq \frac{C_2'}{(k+1)^{1-\bfrak}}.$$
    \item Moreover, $\|x_k-\xhat(y_k)\|$ converges to zero, and the iterates $y_k$ converge to the set $\Yhat$ with probability $1$.
    \end{itemize}
\end{theorem}

\section{Proof Outline}\label{sec:outline}
In this section, we give a rough sketch for the proof for Theorem \ref{thm:main}. Appendix \ref{app:main-proof} presents the complete proof. The proof takes inspiration from analysis of two-time-scale iterations with contractive maps \cite{Doan} and analysis of inexact Krasnoselskii-Mann (KM) iterations \cite{Davis-splitting, Liang-inexact}, but requires novel analysis to deal with the additional errors of two-time-scale updates in a stochastic setting with non-expansive maps. For simplicity, we define $h(y)=g(\xstar(y),y)$. 

The first step in the proof is to show that $\xstar(y)$ is Lipschitz in $y$. 
\begin{lemma}\label{lemma:xstar-lip}
    Suppose Assumptions \ref{f-contrac} and \ref{Lipschitz} hold. Then $\xstar(y)$, the unique fixed point of $f(\cdot,y)$, satisfies 
    $$\|\xstar(y_1)-\xstar(y_2)\|\leq L_0\|y_1-y_2\|,$$
    for all $y_1,y_2\in\RR^{d_2}$. Here $L_0\coloneqq L/(1-\mu)$ is the Lipschitz parameter for the mapping $\xstar(\cdot):\RR^{d_2}\mapsto\RR^{d_1}$.
\end{lemma}
This lemma ensures that the target points for the faster iteration $\xstar(y_k)$ move slowly as the iteration $\{y_k\}$ moves slowly. Apart from the noise sequence $M_{k+1}$, ($\xstar(y_{k+1})-\xstar(y_k)$) acts as another `error term' for the faster iteration. The Lipschitz nature of $\xstar(\cdot)$, along with the slow changes in $y_k$, guarantees that this error term is small. 

The next step of the proof is to obtain an intermediate bound for $\EE[\|x_k-\xstar(y_k)\|^2]$ and $\EE[\|y_k-\ystar\|^2]$. Here $\ystar\in\Ystar$ is some fixed point for the mapping $g(\xstar(\cdot),\cdot)$.
\begin{lemma}\label{lemma:inter}
    Suppose the setting of Theorem \ref{thm:main} holds. Then for $\ystar\in\Ystar$, iterates $\{x_k,y_k\}$ satisfy the following bounds.
\begin{align*}
    \EE\left[\|x_{k+1}-\xstar(y_{k+1})\|^2\mid\FF_k\right]&\leq (1-\mu'\alpha_k)\|x_k-\xstar(y_k)\|^2\nonumber\\
    &\;\;+\Gamma_1\alpha_k^2\left(1+\|x_k-\xstar(y_k)\|^2+\|y_k-\ystar\|^2\right),\\
    \EE\left[\|y_{k+1}-\ystar\|^2\mid\FF_k\right]&\leq \|y_k-\ystar\|^2+\mu'\alpha_k\|x_k-\xstar(y_k)\|^2\nonumber\\
    &\;\;+\Gamma_2\alpha_k^2\left(1+\|x_k-\xstar(y_k)\|^2+\|y_k-\ystar\|^2\right)\\
    &\;\;-\beta_k(1-\beta_k)\|y_k-g(x_k,y_k)\|^2.
\end{align*}
Here $\Gamma_1,\Gamma_2$ are some constants and $\mu'=1-\mu$. 
\end{lemma}
 When deriving the bound for the faster time-scale, we encounter an error term which has a dependence of $\beta_k^2/\alpha_k$ on $k$. The assumption that $\beta_k^2\leq \alpha_k^3$ helps us bound this term by $\alpha_k^2$. Since we have the assumption that $\alpha_k$ is square-summable, we get that this error term is summable. Hence the assumptions $\beta_k^2\leq \alpha_k^3$ and square-summability of $\alpha_k$ are key to show that the iterates are bounded in expectation.
\begin{lemma}\label{lemma:bounded-expectation}
    Suppose the setting of Theorem \ref{thm:main} holds. Then,
    \begin{enumerate}[label=\alph*)]
        \item The iterates are almost surely bounded, i.e., $\sup_k (\|x_k\|+\|y_k\|)<\infty$ a.s.
        \item The iterates $x_k$ and $y_k$ are bounded in expectation. Specifically, there exists $\Gamma_3,\kappa_1>0$ such that for all $k\geq 0$, and for $\ystar\in\Ystar$,
    $$\EE\left[\|x_k-\xstar(y_k)\|^2+\|y_k-\ystar\|^2\right]\leq \Gamma_3,$$
    where $\Gamma_3\coloneqq \kappa_1\left(\kappa_1+\|x_0-\xstar(y_0)\|^2+\|y_0-\ystar\|^2\right).$
    \end{enumerate}
\end{lemma}

The above lemma is central to proving that the iterates converge to the set of fixed points w.p.\ 1. For this, we use the ODE approach for the analysis of SA algorithms. A key requirement for the approach is the almost sure boundedness of iterates, which is established in the above lemma using the Robbins-Siegmund Theorem \cite{Robbins-Siegmund}. The bound for $\EE[\|x_k-\xstar(y_k)\|^2]$ is also direct from the above lemma. We get the following recursive relation using Lemmas \ref{lemma:inter} and \ref{lemma:bounded-expectation}.
$$\EE\left[\|x_{k+1}-\xstar(y_{k+1})\|^2\right]\leq (1-\mu'\alpha_k)\EE\left[\|x_k-\xstar(y_k)\|^2\right]+\Gamma_1(1+\Gamma_3)\alpha_k^2.$$
Solving this recursive relation gives us the bound that $\EE\left[\|x_k-\xstar(y_k)\|^2\right]=\mathcal{O}(\alpha_k)$. On analyzing the other bound, we get that the summation of $\beta_i\EE\left[\|y_i-h(y_i)\|^2\right]$ over all $i$ is bounded by a constant. Recall that $h(y)=g(\xstar(y),y)$.
\begin{lemma}\label{lemma:wp1_and_bound}
    Suppose the setting of Theorem \ref{thm:main} holds. Then,
    \begin{enumerate}[label=\alph*)]
        \item $\|x_k-\xstar(y_k)\|$ converges to zero, and $y_k$ converges to the set $\Ystar$ w.p.\ 1.
        \item There exist constants $C_1,\Gamma_4>0$ such that, for all $k\geq 0$,
        $$\EE[\|x_k-\xstar(y_k)\|^2]\leq \frac{C_1}{(k+1)^{\afrak}}\;\;\text{and}\;\;\sum_{i=0}^k\beta_i\EE\left[\|h(y_i)-y_i\|^2\right]\leq \Gamma_4,$$
        where $C_1=\kappa_2\left(1+\|x_0-\xstar(y_0)\|^2+\|y_0-\ystar\|^2\right)$ for some $\kappa_2>0$.
    \end{enumerate}
\end{lemma}
The above lemma would be sufficient if we wanted to get a bound on the term $\min_{0\leq i\leq k} \EE\left[\|h(y_i)-y_i\|^2\right]$. But we wish to get a bound on $\EE\left[\|h(y_k)-y_k\|^2\right]$. 

For this, we would like to get a bound on $\EE[\|h(y_i)-y_i\|^2]$ in terms of $\EE[\|h(y_k)-y_k\|^2]$. But the martingale noise $M'_{k+1}$ does not allow us to obtain a useful bound. To solve this, we define an averaged error sequence $U_{k+1}=(1-\beta_k)U_k+\beta_kM'_{k+1}$, with $U_0=0$, and a modified iteration of $z_k=y_k-U_k$. We first show that $\EE[\|y_k-z_k\|^2]$ decays at a sufficiently fast rate, implying that we can analyze $\EE[h(z_k)-z_k\|^2]$ instead of $\EE[\|h(y_k)-y_k\|^2].$ Next, we study the iteration for $z_k$, showing that the error sequence for $z_k$ satisfies a similar boundedness condition for the summation as in Lemma \ref{lemma:wp1_and_bound} b). Finally, we fulfill the purpose of studying $z_k$ instead of $y_k$, by obtaining a bound on $\EE[\|h(z_i)-z_i\|^2]$ in terms of $\EE[\|h(z_k)-z_k\|^2]$.
\begin{lemma}\label{lemma:h-results}
    Suppose the setting of Theorem 2.6 holds. Then, for all $k\geq 0$,
    \begin{enumerate}[label=\alph*)]
        \item $\EE[\|y_k-z_k\|^2]=\EE[\|U_k\|^2]\leq \Gamma_5\beta_k$, for some $\Gamma_5>0$.
        \item $\sum_{i=0}^k\beta_i\EE\left[\|h(z_i)-z_i\|^2\right]\leq \Gamma_6$, for some $\Gamma_6>0$.
        \item For all $i\leq k-1$, there exists $\Gamma_7>0$ such that $$\EE\left[\|h(z_i)-z_i\|^2\right]\geq \EE\left[\|h(z_k)-z_k\|^2\right]-\sum_{j=i}^{k-1}\Gamma_7\beta_j\alpha_j.$$
    \end{enumerate}
\end{lemma}
Lemma \ref{lemma:h-results} implies that
$$\EE\left[\|h(z_k)-z_k\|^2\right]\sum_{i=0}^k\beta_i\leq \Gamma_6+\Gamma_7\sum_{i=0}^k\beta_i\sum_{j=i}^{k-1}\beta_j\alpha_j.$$
We finally show that the summation on the right-hand side above is bounded. This requires the assumptions of square-summability of $\alpha_k$ and $\beta_k^2\leq \alpha_k^3$ and completes the proof for Theorem \ref{thm:main}.

\section{Gradient Descent Operators in Slower Time-Scale}\label{sec:gradient-result}
An important example of the iteration in \eqref{iter-main} is one where the slower time-scale (iteration for $y_k$) is a gradient descent iteration in the limit  $x_k\rightarrow \xstar(y_k)$. The following assumption formally states this.
\begin{assumption}\label{assu:G-gradient}
    We assume that there exists differentiable function $J:\RR^{d_2}\mapsto\RR$ such that $g(\xstar(y),y)-y=-\gradJ(y)$. We further assume that $J(\cdot)$ is bounded from below, i.e., there exists $J_{min}\in\RR$ such that $J(y)\geq J_{min}$ for all $y$.
\end{assumption}
Depending on the assumptions on function $J(\cdot)$, different types of convergence rates can be obtained. The function $J(\cdot)$ is smooth in all cases (i.e., $\gradJ(\cdot)$ is Lipschitz).
\subsubsection*{$J(\cdot)$ is strongly convex}In this case, the function $g(\xstar(\cdot),\cdot)$ can be shown to be contractive, and hence we have a two-time-scale iteration where both time-scales have contractive mappings. In this case, the result from \cite{Chandak-TTS-Opti} can directly be applied, to obtain a bound of $\mathcal{O}(1/k)$ on $\EE[\|x_k-\xstar(y_k)\|^2+\|y_k-\ystar\|^2]$. Here $\ystar$ is the unique minimizer for $J(\cdot)$.
\subsubsection*{$J(\cdot)$ is convex} In this case, the function $g(\xstar(\cdot),\cdot)$ is non-expansive. This gives us a two-time-scale iteration, which satisfies the assumptions for Theorem \ref{thm:main}. We can apply our result to get a bound of $\mathcal{O}(1/k^{0.25-\epsilon})$ on $\EE[\|x_k-\xstar(y_k)\|^2+\|\gradJ(y_k)\|^2]$. This case is discussed in more detail in Subsection \ref{subsec:minimax}, when we discuss minimax optimization.
\subsubsection*{$J(\cdot)$ is non-convex} In this case, the function $g(\xstar(\cdot),\cdot)$ is not necessarily non-expansive, and hence Theorem \ref{thm:main} can not be applied. But under Lipschitz assumptions on function $g(\cdot,\cdot)$, the function $J(\cdot)$ is smooth, and hence bounds can be obtained in this case using a proof technique similar to the proof for Theorem \ref{thm:main}. 
Consider the following assumption. 
    \begin{assumption}\label{assu:grad-all}
        We assume that the following statements hold.
        \begin{enumerate}[label=\alph*)]
            \item Function $f(\cdot,y)$ is contractive with parameter $\mu$ (Assumption \ref{f-contrac}).
            \item Function $f$ and $g$ are Lipschitz with parameter $L$ (Assumption \ref{Lipschitz}).
            \item $M_{k+1}$ and $M'_{k+1}$ are martingale difference sequences with bounded variances, i.e., $\EE[\|M_{k+1}\|^2+\|M'_{k+1}\|^2\mid\FF_k]\leq \cc_2$.
            \item The stepsize sequences satisfy $\afrak<\bfrak<1$ and $0.5<\bfrak$. We also have the key assumption that $\beta_k^2\leq \alpha_k^3$ and $\beta_k\alpha_k\leq\beta\alpha/(k+K_1)$.
        \end{enumerate}
    \end{assumption}
    The key difference in assumptions is that we \textbf{no longer require $\alpha_k$ to be square summable}, allowing sequences such as $\alpha_k=\mathcal{O}(1/k^{0.4})$ and $\beta_k=\mathcal{O}(1/k^{0.6})$. Then, we have the following theorem, with its proof presented in Appendix \ref{app:gradient-proof}.  
    \begin{theorem}\label{thm:gradient}
    Suppose that Assumptions \ref{assu:G-gradient} and \ref{assu:grad-all} hold. 
    Then there exist constants $C_3,C_4$ such that if $\alpha\leq C_4$, then for all $k\geq 0$,
  $$\min_{0\leq i\leq k} \EE\left[\|\gradJ(y_i)\|^2\right]\leq \frac{C_3(1+\log(k+1))}{(1+k)^{1-\bfrak}}.$$
    \end{theorem}
    The optimal rate in this case is obtained at $\afrak=0.4$ and $\bfrak=0.6$, with a bound of $\min_{0\leq i\leq k} \EE\left[\|\gradJ(y_i)\|^2\right]=\mathcal{O}(\log(k)/k^{0.4})$. In this case, we do not have a bound on the gradient at time $k$, but only have a bound on the minimum gradient attained before time $k$. Such non-convex two-time-scale algorithms arise in non-convex minimax optimization and reinforcement learning (actor-critic algorithms). The same upper bound of $\mathcal{O}(\log(k)/k^{0.4})$ was obtained in \cite{Doan-actor-critic, Wu-actor-critic} under different assumptions on the noise sequence, and remains the best-known upper bound for this setting.

\section{Applications}\label{sec:applications}
In this section, we discuss three applications of the framework and results discussed in Section \ref{sec:main}. 
\subsection{Minimax Optimization}\label{subsec:minimax}
Consider the following minimax optimization problem:
$$\min_{y\in\RR^{d_2}}\max_{x\in\XX} H(x,y),$$
where $H:\XX\times \RR^{d_2}\mapsto \RR$ is the objective function. This formulation of finding saddle point solutions to minimax optimization problems arises in several areas, including two-player zero-sum games, generative adversarial networks, and robust statistics \cite{jordan-ttsgda}. We consider the Two-Time-scale Stochastic Gradient Descent Ascent (TTSGDA) algorithm \cite{jordan-ttsgda} for finding saddle points. 
\begin{equation}\label{iter-minimax}
    \begin{aligned}
    &x_{k+1}=\PP_\XX\left(x_k+\alpha_k(\nabla_x H(x_k,y_k)+M_{k+1})\right)\\
    &y_{k+1}=y_k+\beta_k(-\nabla_y H(x_k,y_k)+M'_{k+1}).
\end{aligned}
\end{equation}
Here $M_{k+1}$ and $M'_{k+1}$ denote the noise arising from taking gradient samples at the current iterate. These noise sequences are assumed to be unbiased and have bounded variance, i.e., they satisfy Assumption \ref{Martingale}. 

We work under the setting of smooth strongly concave-convex minimax optimization. This is formally stated in the following assumption.
\begin{assumption}\label{assu:minimax}
    Function $H(x,y)$ is $L-$smooth, $\rho-$strongly concave in $x$ and convex in $y$. The set $\XX\subset \RR^{d_1}$ is a convex set with a bounded diameter. 
\end{assumption}
This assumption is sufficient to ensure that the faster time-scale (update of iterate $x_k$) has a contractive mapping and the slower time-scale (update of iterate $y_k$) has a non-expansive mapping. Define $\Phi(y)=\max_{x\in\RR^{d_1}}H(x,y)$. Then we have the following result, with its proof presented in Appendix \ref{app:applications-proof}.
\begin{corollary}\label{coro:minimax}
    Suppose Assumption \ref{assu:minimax} holds for iterates $\{x_k,y_k\}$ defined in \eqref{iter-minimax}, and let $\alpha_k=\mathcal{O}\left(1/(k+1)^{0.5+0.66\epsilon}\right)$ and $\beta_k=\mathcal{O}\left(1/(k+1)^{0.75+\epsilon}\right)$. Then the iterates $\{x_k,y_k\}$ converge to the set of saddle points of the function $H(x,y)$ a.s., and $\EE\left[\|\nabla \Phi(y_k)\|^2\right]\leq \mathcal{O}(1/k^{1/4-\epsilon}).$ Here $\epsilon$ can be arbitrarily small.
\end{corollary}

We note that stronger guarantees can be obtained for strongly convex-concave minimax optimization by using techniques such as extra-gradient, proximal best response, or just by considering averaged iterates \cite{minimax-other-1,minimax-other-2}, but to the best of our knowledge, our result is the first last-iterate bound for \textit{vanilla} TTSGDA.

\subsection{Linear Stochastic Approximation}\label{subsec:linear}
Consider the problem of finding solutions for the following set of linear equations:
\begin{equation}\label{system-linear}
    A_{11}x+A_{12}y=b_1\quad\text{and}\quad A_{21}x+A_{22}y=b_2.
\end{equation}
At each time $k$, we obtain unbiased estimates $\tilde{A}_{ij}^{(k+1)}$ and $\tilde{b}_i^{(k+1)}$ of each $A_{ij}$ and $b_i$, $i,j\in\{1,2\}$. This problem is solved using a linear two-time-scale algorithm.
\begin{equation}\label{iter-linear}
    \begin{aligned}
    &x_{k+1}=x_k+\alpha_k\left(\tilde{b}_1^{(k+1)}-\tilde{A}_{11}^{(k+1)}x_k-\tilde{A}_{12}^{(k+1)}y_k\right)\\
    &y_{k+1}=y_k+\beta_k\left(\tilde{b}_2^{(k+1)}-\tilde{A}_{21}^{(k+1)}x_k-\tilde{A}_{22}^{(k+1)}y_k\right).
\end{aligned}
\end{equation}
This algorithm finds use in reinforcement learning \cite{Shaan} and linear SA with Polyak averaging \cite{Konda}. Define matrix $\Delta=A_{22}-A_{21}A_{11}^{-1}A_{12}$. Under the assumption that matrices $-A_{11}$ and $-\Delta$ are Hurwitz, i.e., the real part of all their eigenvalues is negative, iteration \eqref{iter-linear} can be written as two-time-scale SA with both time-scales as fixed point iterations with contractive mappings.

In this paper, we consider the case where the matrix $-A_{11}$ is negative definite and the matrix $-\Delta$ is negative semidefinite. Formally, we make the following assumption.
\begin{assumption}\label{assu:linear}
    Iteration \eqref{iter-linear} satisfies the following.
    \begin{enumerate}[label=\alph*)]
        \item The matrix $-A_{11}$ is negative definite and the matrix $-\Delta$ is (non-zero) negative semidefinite. 
        \item There exists a solution for the system of equations in \eqref{system-linear}.
        \item For $i,j\in\{1,2\}$, $\EE\left[\tilde{A}_{ij}^{(k+1)}\mid\FF_k\right]=A_{ij},\EE\left[\tilde{b}_i^{(k+1)}\mid \FF_k\right]=b_i.$ Moreover, $$\EE\left[\|\tilde{A}_{ij}^{(k+1)}-A_{ij}\|^2\mid\FF_k\right]\leq \cc_3,\quad\text{and}\quad\EE\left[\|\tilde{b}_i^{(k+1)}-b_i\|^2\mid\FF_k\right]\leq \cc_3.$$
    \end{enumerate}
\end{assumption}
Part a) of this assumption ensures that iteration \eqref{iter-linear} satisfies Assumptions \ref{f-contrac} and \ref{g-nonexp}. As Hurwitz matrices also correspond to contractive fixed point iterations \cite{Konda}, we can also work with Hurwitz $-A_{11}$ (instead of negative definite). Here is the result for this setting, with its proof presented in Appendix \ref{app:applications-proof}.
\begin{corollary}\label{coro:linear}
    Suppose Assumption \ref{assu:linear} holds for iterates $\{x_k,y_k\}$ defined in \eqref{iter-linear}, and let $\alpha_k=\mathcal{O}\left(1/(k+1)^{0.5+0.66\epsilon}\right)$ and $\beta_k=\mathcal{O}\left(1/(k+1)^{0.75+\epsilon}\right)$. Then the iterates $\{x_k,y_k\}$ converge to the set of solutions of system of equations \eqref{system-linear} a.s., and 
    $$\EE\left[\|x_k-\xstar(y_k)\|^2\right]\leq\mathcal{O}\left(\frac{1}{k^{0.5}}\right)\;\;\text{and}\;\;\EE\left[\|A_{21}x_k+A_{22}y_k-b_2\|^2\right]\leq \mathcal{O}\left(\frac{1}{k^{0.25-\epsilon}}\right).$$ Here $\epsilon$ can be arbitrarily small, and $\xstar(y)=A_{11}^{-1}(b_1-A_{12}y)$.
\end{corollary}
We note that the linear non-expansive map can also be treated as a linear contractive map under the seminorm $\|x\|_\Delta \coloneqq \sqrt{x^\top \Delta x}$, which vanishes on $\ker(\Delta)$. We believe that stronger guarantees can be obtained for iteration~\eqref{iter-linear} using existing analyses for seminorm-contractive stochastic approximation \cite{seminorm}, although such analyses are currently available only for the single-time-scale setting and would need to be extended to the two-time-scale case. We include this application to show the broad applicability of our framework.

\subsection{Constrained Optimization with Lagrangian Multipliers}\label{subsec:lagrangian}
Consider the following constrained optimization problem for $x\in\RR^{d_1}$:
\begin{align}
    &\text{maximize} \;\;H_0(x)\nonumber\\
    &\text{subject to:}\nonumber\\
    &\;\;\;\;\;H_i(x)\leq0, i=1,\ldots,m\nonumber\\
    &\;\;\;\;\;Ax=b_0.
\end{align}
Here the constraints $H_i(x)\leq 0$ together represent the set $\XX$, i.e., $\XX=\{x\mid H_i(x)\leq 0, i=1,\ldots,m\}.$ The equation $Ax=b_0$ for $A\in\RR^{d_2\times d_1}$ and $b_0\in\RR^{d_2}$ represents the additional linear constraint. This problem can be solved using Lagrange multipliers, specifically, a two-time-scale Lagrangian optimization method \cite{two-time-lagrangian}.

Consider the following iteration.
\begin{equation}\label{iter-Lagrangian}
    \begin{aligned}
    &x_{k+1}=\PP_\XX\big(x_k+\alpha_k(\nabla H_0(x_k)-A^T\lambda_k+M_{k+1})\big)\\
    &\lambda_{k+1}=\lambda_k+\beta_k(Ax_k-b_0).
\end{aligned}
\end{equation}
Here $\lambda_k\in\RR^{d_2}$ denotes the Lagrange multiplier. $M_{k+1}$ denotes the noise arising from taking gradient samples at the current value of $x_k$, and satisfies Assumption \ref{Martingale}. 

This problem can be viewed as a special case of strongly convex--concave minimax optimization. We treat it separately due to the interesting role played by projections (Lemma~\ref{lemma:Lagrangian}), as well as its practical applications. The algorithm is particularly useful in distributed settings, where each node or agent updates some part of the variable $x$ under some local convex constraints, and there is a global linear constraint. In that case, the local updates only require knowledge of the gradient, the Lagrange multiplier, and their local constraints. Another similar setting in which a similar algorithm can be used is generalized Nash equilibrium problems (or Nash equilibrium under constraints) \cite{GNEP} and game control for strongly monotone games \cite{Chandak}. In this setting, each player takes action in a convex set and updates their action using gradient ascent. The game controller has a global linear objective which they want the player's actions to satisfy \cite{Chandak}. 

We make the following assumption. The second part of the assumption is required for the existence of $\lambda$ which satisfies $A\xhat(\lambda)=b_0$.
\begin{assumption}\label{assu:Lagrangian}
    Iteration \eqref{iter-Lagrangian} satisfies the following statements.
    \begin{enumerate}[label=\alph*)]
        \item Function $H_0(x)$ is $\rho-$strongly concave and $L$-smooth.
        \item Functions $H_i(x)$ are continuously differentiable and convex. Moreover, there exists $x\in\text{int}(\XX)$ (i.e., $H_i(x)<0, i=1,\ldots,m$) such that $Ax=b_0$.
    \end{enumerate}
\end{assumption}

This problem is an example of a setting where the projection in the faster time-scale gives rise to the non-expansiveness in the slower time-scale. Recall the definitions of $\xstar(y)$ and $\xhat(y)$. In this case, $\xstar(\lambda)$ is the unique solution to $\nabla H_0(x)-A^T\lambda=0$ and $\xhat(\lambda)$ is the unique solution to $x=\PP_\XX(x+\nabla H_0(x_k)-A^T\lambda)$. Then we have the following lemma.
\begin{lemma}\label{lemma:Lagrangian}
    Under Assumption \ref{assu:Lagrangian} and the assumption that the matrix $A$ has full row rank:
    \begin{enumerate}[label=\alph*)]
        \item $-A\xstar(\lambda)$ is strongly monotone.
        \item $-A\xhat(\lambda)$ is co-coercive.
    \end{enumerate}
\end{lemma}
The above lemma implies that in the absence of the projection onto set $\XX$, iteration \eqref{iter-Lagrangian} would have contractive mappings in both time-scales. On the other hand, the slower time-scale has a non-expansive mapping in the presence of the projection. The following is our main result for this setting.
\begin{corollary}\label{coro:Lagrangian}
    Suppose Assumption \ref{assu:Lagrangian} holds for iterates $\{x_k,\lambda_k\}$ defined in \eqref{iter-Lagrangian}, and let $\alpha_k=\mathcal{O}\left(1/(k+1)^{0.5+0.66\epsilon}\right)$ and $\beta_k=\mathcal{O}\left(1/(k+1)^{0.75+\epsilon}\right)$. Then $x_k$ converges to $\xhat(\lambda_k)$ and $Ax_k$ converges to $b_0$ a.s., and $$\EE\left[\|x_k-\xhat(y_k)\|^2\right]\leq\mathcal{O}\left(\frac{1}{k^{0.5}}\right)\;\;\text{and}\;\;\EE\left[\|Ax_k-b_0\|^2\right]\leq \mathcal{O}\left(\frac{1}{k^{0.25-\epsilon}}\right).$$ Here $\epsilon$ can be arbitrarily small.
\end{corollary}

\section{Numerical Experiments}\label{sec:experiments}
In this section, we simulate the three applications from the previous section. Each plot represents the average over 200 independent runs of the algorithm. The shaded region indicates one standard deviation around the mean. The stepsize sequences are of the form $\alpha_k=\alpha/(k+100)^{\afrak}$ and $\beta_k=\beta/(k+100)^{\bfrak}$. 

In Figure \ref{fig:linear}, we simulate the linear stochastic approximation setting. We generate a random positive definite matrix $A_{11}\in\RR^{20\times 20}$ and generate a positive semi-definite matrix $\Delta\in\RR^{20\times 20}$ of rank $5$. We set $A_{12}=A_{21}=I$ and $A_{22}=A_{11}^{-1}+\Delta$. The vectors $b_1, b_2\in\RR^{20}$ are obtained by randomly generating $x',y'\in\RR^{20}$ and computing $b_1=A_{11}x'+A_{12}y'$ and $b_2=A_{21}x'+A_{22}y'$. At each timestep, i.i.d.\ normal noise is added to the vectors and the matrices, and the noisy residual error is used for updating the iterates $x_k$ and $y_k$. 

In Figure \ref{fig:linear_1}, we study the effect of the exponent in the stepsize (parameters $\afrak$ and $\bfrak$) on the residual error. We see that a slowly decreasing stepsize has a better rate of convergence. This matches our intuition and the bound obtained in Theorem \ref{thm:main}. In Figure \ref{fig:linear_2}, we study the effect of constants $\alpha,\beta$ on the residual error. As before, larger stepsizes accelerate convergence. Interestingly, the residual error appears to be primarily influenced by the value of $\beta$, with minimal sensitivity to $\alpha$.

\begin{figure}[htbp]
    \centering
    \begin{subfigure}[b]{0.495\textwidth}
        \centering
        \includegraphics[width=\linewidth]{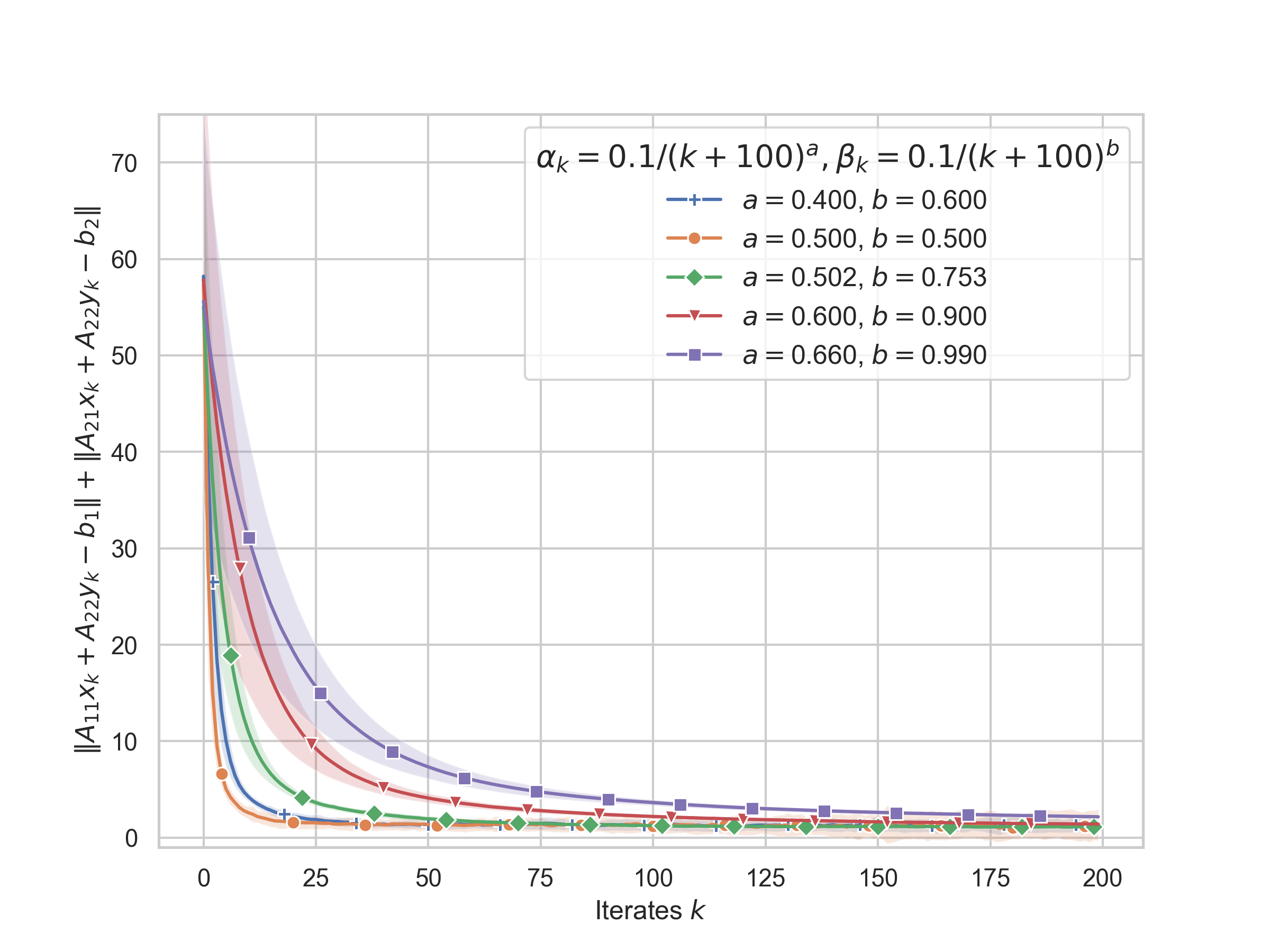}
        \caption{Variation with exponents $\afrak,\bfrak$}
        \label{fig:linear_1}
    \end{subfigure}
    \begin{subfigure}[b]{0.495\textwidth}
        \centering
        \includegraphics[width=\linewidth]{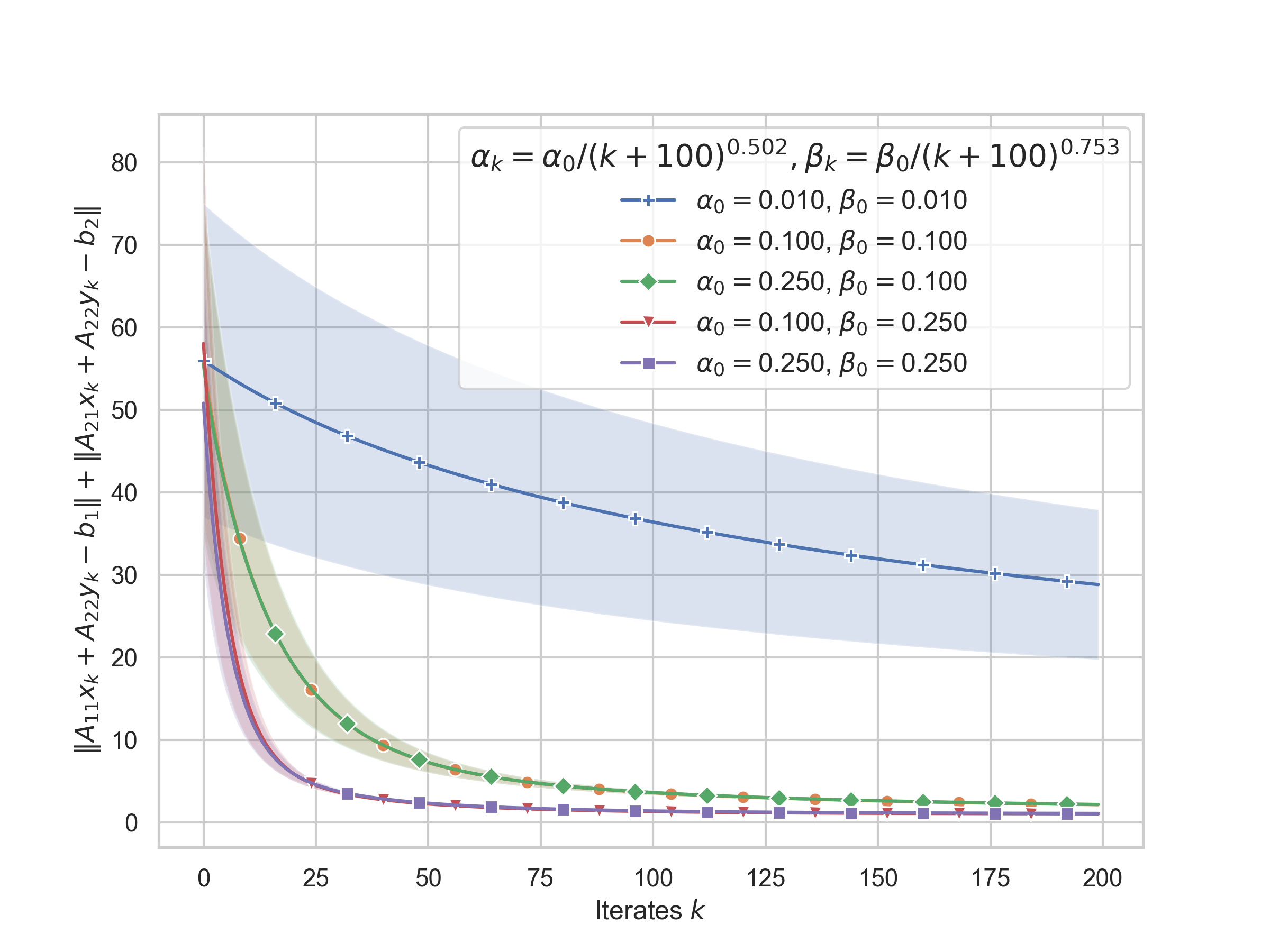}
        \caption{Variation with constants $\alpha,\beta$}
        \label{fig:linear_2}
    \end{subfigure}
    \caption{Linear Stochastic Approximation: effect of stepsize sequence on residual errors}
    \label{fig:linear}
\end{figure}

We also want to note an interesting pattern observed consistently across all three scenarios. Even though the choices $\afrak=0.5, \bfrak=0.5$ and $\afrak=0.4,\bfrak=0.6$ do not satisfy the assumptions required for our theoretical bounds, they show the the fastest empirical convergence in Figures~\ref{fig:linear}–\ref{fig:constrained}. This may indicate that stronger bounds may be possible under modified assumptions. However, these settings also exhibit sensitivity to parameter perturbations: for instance, increasing $\alpha$ from $0.1$ to $0.5$ can lead to instability in some runs, with iterates diverging to magnitudes exceeding $10^{200}$.

\begin{figure}[htbp]
    \centering
    \begin{subfigure}[b]{0.495\textwidth}
        \centering
        \includegraphics[width=\linewidth]{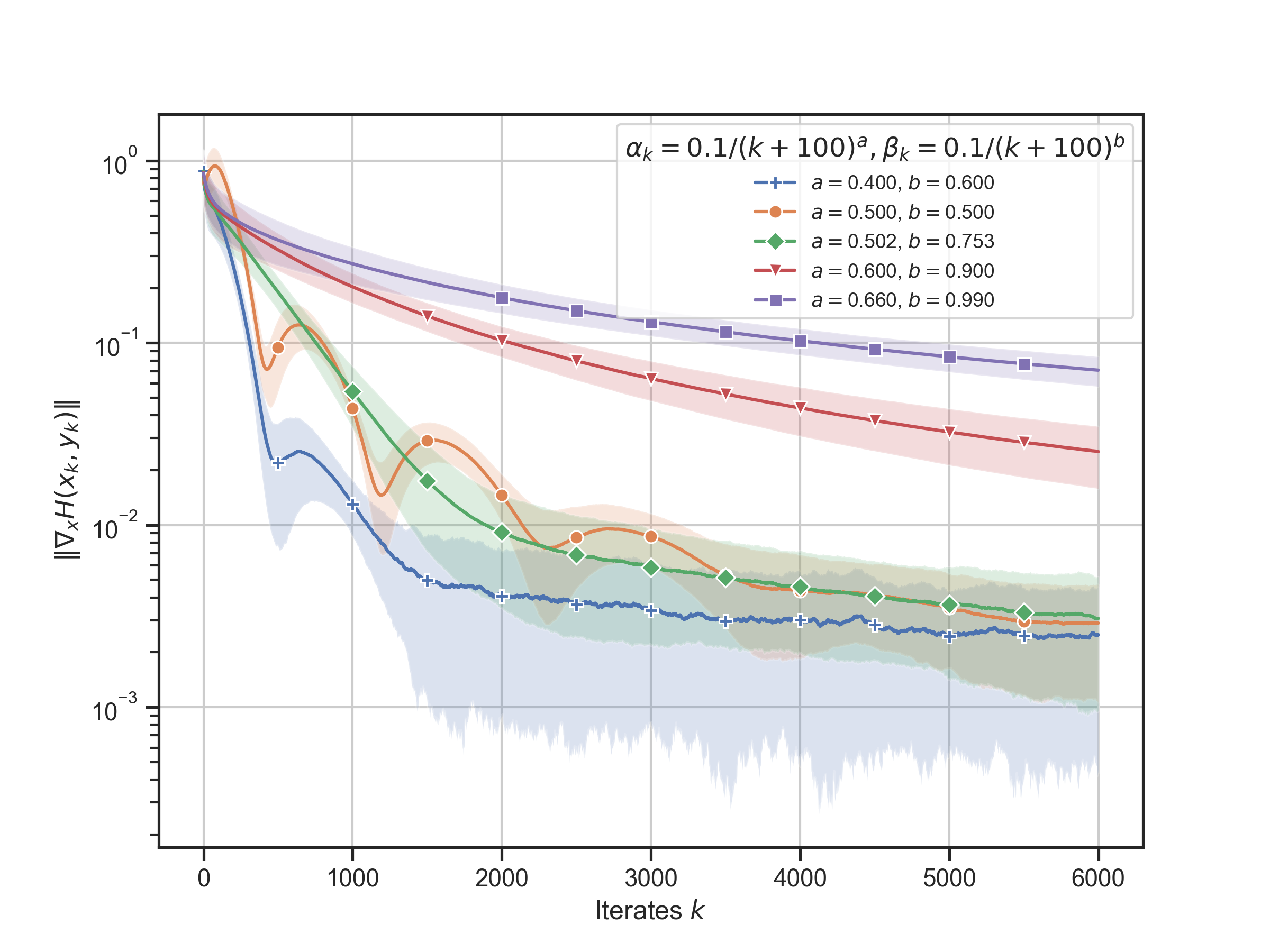}
        \caption{Residual error for the faster time-scale}
        \label{fig:minimax_1}
    \end{subfigure}
    \hfill
    \begin{subfigure}[b]{0.495\textwidth}
        \centering
        \includegraphics[width=\linewidth]{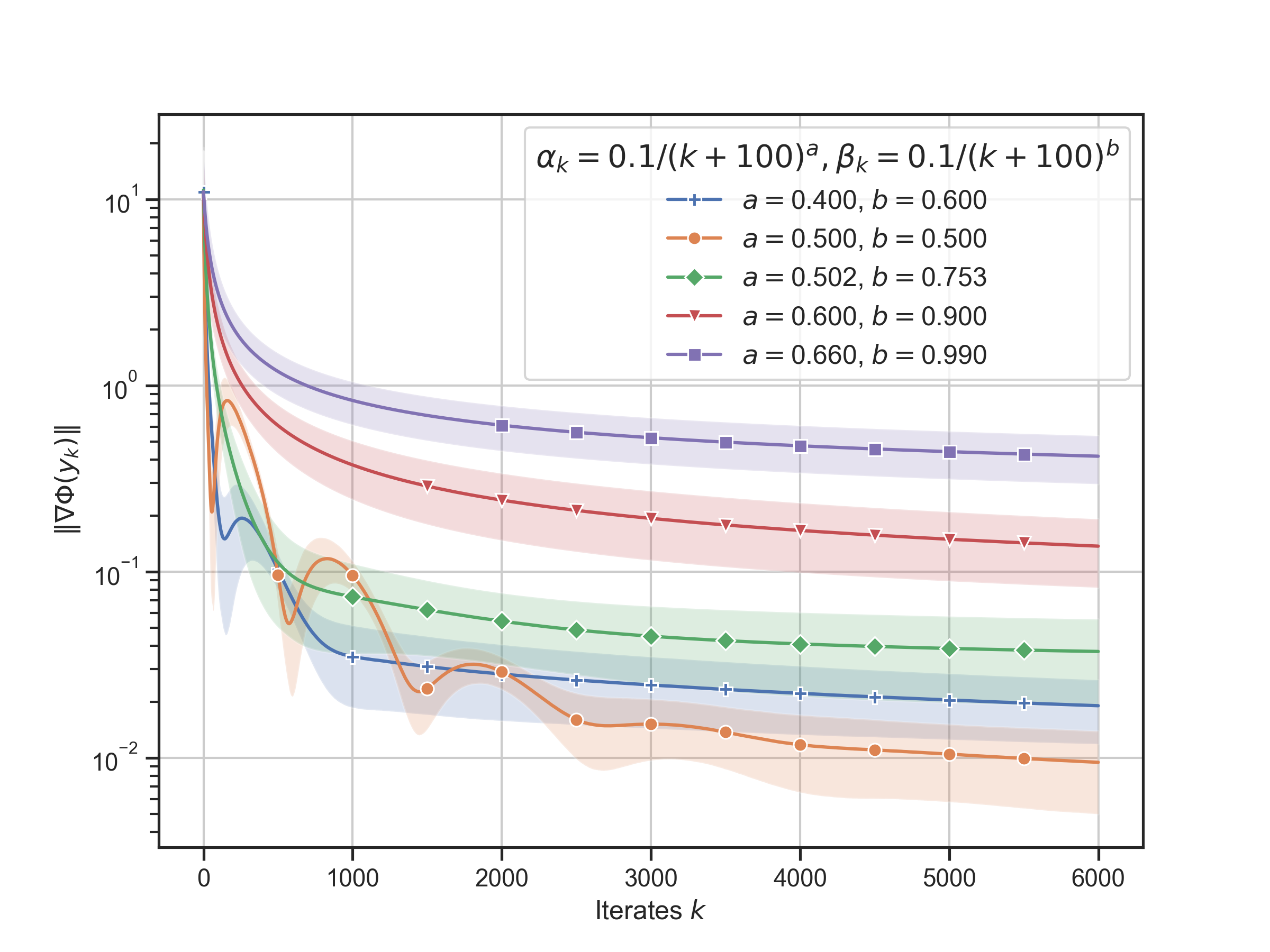}
        \caption{Residual error for the slower time-scale}
        \label{fig:minimax_2}
    \end{subfigure}
    \caption{Minimax Optimization: residual errors for the two time-scales}
    \label{fig:minimax}
\end{figure}

We present the simulations for minimax optimization in Figure \ref{fig:minimax}. We choose the function $H(x,y)=x^TAy-0.5\|x\|^2+(y^TQy)^2$ where $A\in\RR^{5\times 5}$ is a randomly generated matrix and $Q\in\RR^{5\times5}$ is a randomly generated symmetric positive semi-definite matrix. In Figure \ref{fig:minimax_1}, we plot $\|\nabla_x H(x_k,y_k)\|$, the residual error for the faster time-scale and in Figure \ref{fig:minimax_2}, we plot $\|\nabla \Phi(y_k)\|$, the residual error for the slower time-scale. We observe that both these residual errors converge to zero, and the dependence on the stepsize is similar to that observed in the linear case. 

For constrained optimization, we choose the function $H_0(x)=\|x-\ell\|^2+\sum_{i=1}^{20}e^{x(i)}$ where $x,\ell\in\RR^{20}$ and $x(i)$ denotes the $i$-th element of the vector. Motivated by the application to distributed optimization, we impose local constraints of the form $H_j(x) \leq 0$ for each $j \in {1, \ldots, 4}$, where $H_j(x) = \sum_{i=5j+1}^{5j+5} x(i)^2 - 4$. These constraints enforce that each 5-dimensional block of the vector $x$ lies within a ball of radius two. Vectors $b_0,\ell$ and the matrix $A$ are chosen randomly to ensure that all assumptions are satisfied. In Figure \ref{fig:constrained_1}, we study the effect of the stepsize constants $\alpha$ and $\beta$ on the residual error. Since the faster time-scale is projected, larger stepsizes can be used without causing instability in the iterates. However, while increasing the stepsize initially improves convergence rate, further increases eventually degrade performance. 

\begin{figure}[htbp]
    \centering
    \begin{subfigure}[b]{0.495\textwidth}
        \centering
        \includegraphics[width=\linewidth]{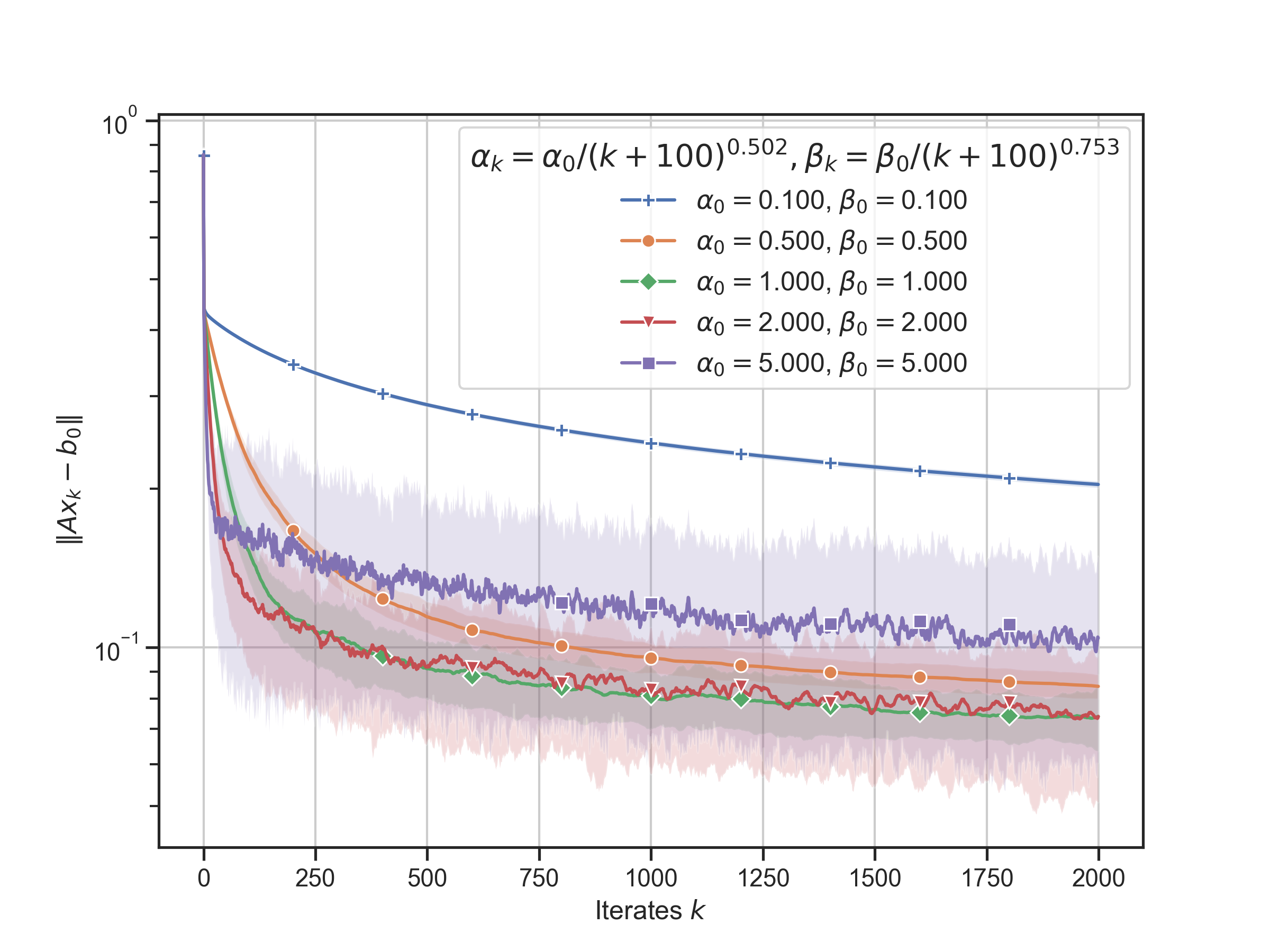}
        \caption{Effect of constants $\alpha,\beta$}
        \label{fig:constrained_1}
    \end{subfigure}
    \hfill
    \begin{subfigure}[b]{0.495\textwidth}
        \centering
        \includegraphics[width=\linewidth]{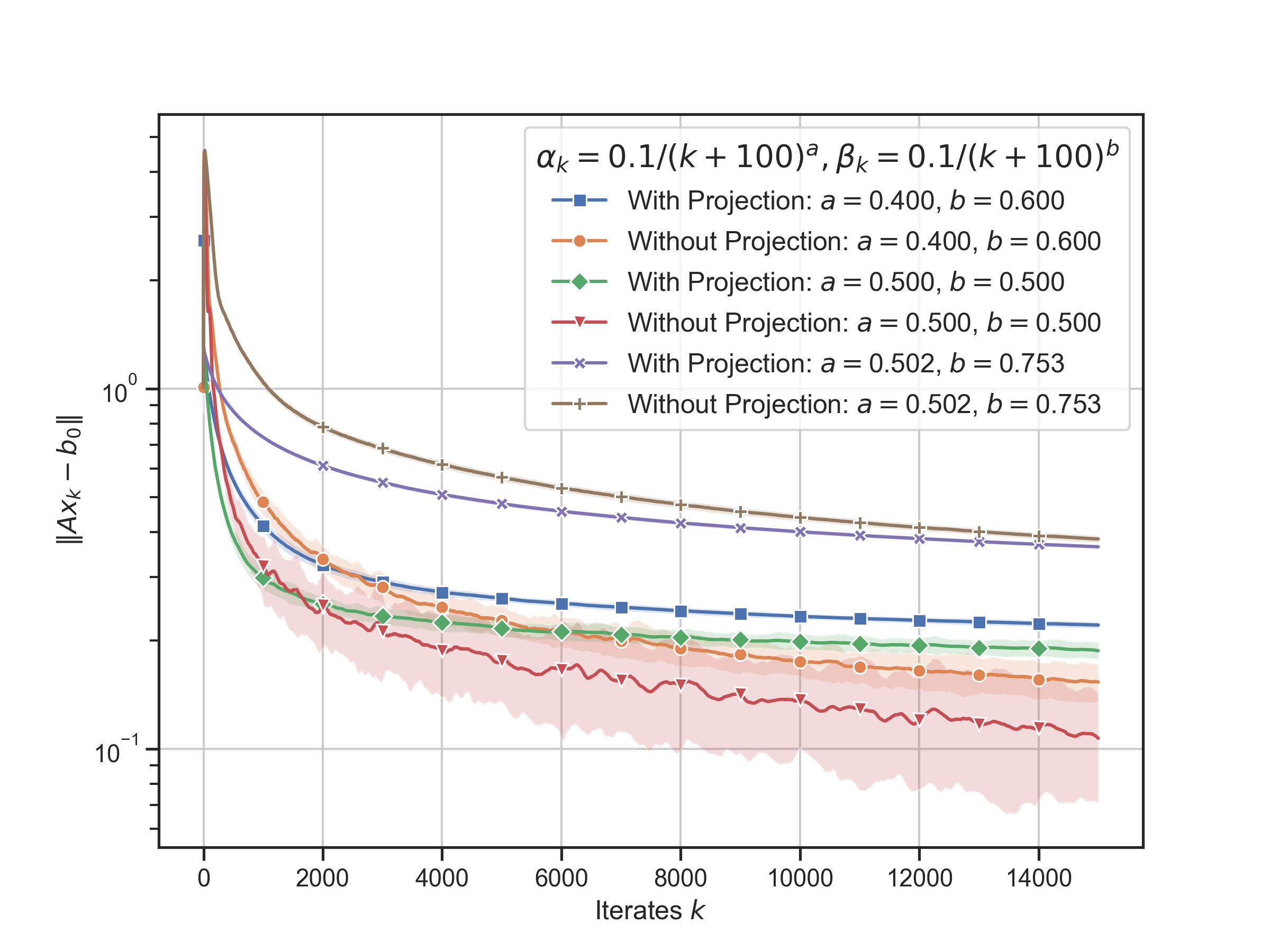}
        \caption{Comparison with the non-projected case}
        \label{fig:constrained_2}
    \end{subfigure}
    \caption{Constrained optimization: (a) effect of stepsize on residual error, and (b) comparison with the case where there is no projection in the faster time-scale}
    \label{fig:constrained}
\end{figure}

In Figure \ref{fig:constrained_2}, we compare the performance of the projected case with the case where there are no projections in the faster time-scale. Recall that the projection in the faster time-scale leads to the non-expansiveness in the slower time-scale and in the absence of a projection, both time-scales have contractive mappings (Lemma \ref{lemma:Lagrangian}). We observe that the no-projection case exhibits faster convergence, aligning with the intuition that contractive mappings lead to improved convergence rates (Table~\ref{table:comp}).

\section{Conclusion}\label{sec:conclusion}
In this paper, we studied two-time-scale iterations where the slower time-scale has a non-expansive mapping. Our main contribution is the finite-time bound for this setting. We showed that the mean square residual error decays at a rate of $\mathcal{O}(k^{-1/4+\epsilon})$, where $\epsilon>0$ is arbitrarily small depending on the step size. Our analysis takes inspiration from techniques in contractive SA and inexact KM iterations. Extensions of our work include studying the same algorithm under Markovian noise and extending the result to Banach spaces. Another extension could be to study non-expansive mappings in both time-scales, but the analysis would require further assumptions to ensure that the faster time-scale has a unique fixed point for each $y$.

\appendix
\section{Proofs required for Theorem \ref{thm:main}}\label{app:main-proof}

\subsection{Proof for Lemma \ref{lemma:xstar-lip}}

\begin{proof}[\textbf{Proof for Lemma \ref{lemma:xstar-lip}}]
    Note that 
    \begin{align*}
        \|\xstar(y_1)-\xstar(y_2)\|&=\|f(\xstar(y_1),y_1)-f(\xstar(y_2),y_2)\|\\
        &\leq \|f(\xstar(y_1),y_1)-f(\xstar(y_2),y_1)\|+\|f(\xstar(y_2),y_1)-f(\xstar(y_2),y_2)\|\\
        &\leq \mu\|\xstar(y_1)-\xstar(y_2)\|+L\|y_1-y_2\|. 
    \end{align*}
    Here the first equality follows from the fact that $\xstar(y)$ is a fixed point for $f(\cdot,y)$ and the last inequality follows from contractiveness and Lipschitzness of $f(x,y)$ in $x$ and $y$, respectively. This implies that 
    $(1-\mu)\|\xstar(y_1)-\xstar(y_2)\|\leq L\|y_1-y_2\|.$
    Hence $\xstar(\cdot)$ is Lipschitz with constant $L_0\coloneqq L/\mu'$, where $\mu'=1-\mu$. 
\end{proof}

\subsection{Proof for Lemma \ref{lemma:inter}}

\begin{proof}[\textbf{Proof for Lemma \ref{lemma:inter}}] For this proof, we first give an intermediate bound on $\EE[\|x_{k+1}-\xstar(y_{k+1})\|^2]$ and then on $\EE[\|y_{k+1}-\ystar\|^2]$. 
\subsubsection*{Recursive Bound on $\EE[\|x_{k+1}-\xstar(y_{k+1})\|^2]$} By adding and subtracting $\xstar(y_k)$,
    \begin{subequations}\label{split1}
        \begin{align}
            &\|x_{k+1}-\xstar(y_{k+1})\|^2\nonumber\\
            &=\|x_k-\xstar(y_k)+\alpha_k(f(x_k,y_k)-x_k)\|^2\label{split11}\\
            &\;\;+\|\alpha_kM_{k+1}+\xstar(y_k)-\xstar(y_{k+1})\|^2\label{split12}\\
            &\;\;+2\Big\langle x_k-\xstar(y_k)+\alpha_k(f(x_k,y_k)-x_k),\alpha_kM_{k+1}+\xstar(y_k)-\xstar(y_{k+1})\Big\rangle. \label{split13}
        \end{align}
    \end{subequations}
 We deal with the three terms \eqref{split11}-\eqref{split13} separately as follows.

    \textbf{Term \ref{split11} --- } 
    \begin{align}\label{split11-simp}
        &\|x_k-\xstar(y_k)+\alpha_k(f(x_k,y_k)-x_k)\|^2\nonumber\\
        &=\|(1-\alpha_k)(x_k-\xstar(y_k))+\alpha_k(f(x_k,y_k)-f(\xstar(y_k),y_k))\|^2\nonumber\\
        &\leq \big((1-\alpha_k)\|x_k-\xstar(y_k)\|+\alpha_k\mu\|x_k-\xstar(y_k)\|\big)^2\nonumber\\
        &\leq (1-\mu'\alpha_k)^2\|x_k-\xstar(y_k)\|^2.
        \end{align}
 Here, the first equality follows from the fact that $f(\xstar(y_k),y_k)=\xstar(y_k)$, and the first inequality follows from the property that $f(\cdot,y)$ is contractive, and that $0\leq \alpha_k\leq 1$.

        \textbf{Term \ref{split12} --- } This term can be thought of as the  `error term' for the iteration with the faster time-scale. We simplify it as follows. 
        \begin{align}\label{split12-simp}
            &\EE[\|\alpha_kM_{k+1}+\xstar(y_k)-\xstar(y_{k+1})\|^2\mid\FF_k]\nonumber\\
            &\leq \EE\left[2\alpha_k^2\|M_{k+1}\|^2+2\|\xstar(y_{k+1})-\xstar(y_k)\|^2\mid\FF_k\right]\nonumber\\
            &\leq 2\cc_2\alpha_k^2(1+\|x_k\|^2+\|y_k\|^2)+2L_0^2\zeta_1\beta_k^2(1+\|x_k\|^2+\|y_k\|^2)\nonumber\\
            &= (2\cc_2\alpha_k^2+2L_0^2\zeta_1\beta_k^2)(1+\|x_k\|^2+\|y_k\|^2),
        \end{align}
        where the second inequality uses Assumption \ref{Martingale} and Lemma \ref{lemma:tech}\ref{lemma:yn-diff} (Appendix \ref{app:tech}).
        
        \textbf{Term \ref{split13} --- } To simplify \eqref{split13}, we first note that:
        \begin{align*}
            &2\Big\langle x_k-\xstar(y_k)+\alpha_k(f(x_k,y_k)-x_k),\xstar(y_k)-\xstar(y_{k+1})\Big\rangle\\
            &\leq \mu'\alpha_k\|x_k-\xstar(y_k)+\alpha_k(f(x_k,y_k)-x_k)\|^2+\frac{1}{\mu'\alpha_k}\|\xstar(y_k)-\xstar(y_{k+1})\|^2.
        \end{align*}
        This follows from applying Cauchy-Schwarz inequality, and then applying AM-GM inequality - $2ab\leq \vartheta a^2+(1/\vartheta)b^2$, with $\vartheta=\mu'\alpha_k$. We apply \eqref{split11-simp} and Lemma \ref{lemma:tech}\ref{lemma:yn-diff}, respectively, for the two terms.
        Using the fact that $M_{k+1}$ is a martingale difference sequence, we have 
        $\EE\Big[2\big\langle x_k-\xstar(y_k)+\alpha_k(f(x_k,y_k)-x_k),M_{k+1}\big\rangle\mid \FF_k\Big]=0.$
        Hence,
        \begin{align}\label{split13-simp}
            &\EE\Big[2\big\langle x_k-\xstar(y_k)+\alpha_k(f(x_k,y_k)-x_k),\alpha_kM_{k+1}+\xstar(y_k)-\xstar(y_{k+1})\big\rangle\mid\FF_k\Big]\nonumber\\
            &\leq \mu'\alpha_k\|x_k-\xstar(y_k)\|^2+\frac{L_0^2\zeta_1}{\mu'}\frac{\beta_k^2}{\alpha_k}(1+\|x_k\|^2+\|y_k\|^2).
        \end{align}

Having obtained a bound on the conditional expectation of the three terms in \eqref{split1}, we now combine the bounds \eqref{split11-simp}-\eqref{split13-simp} to get
\begin{align*}
    \EE\left[\|x_{k+1}-\xstar(y_{k+1})\|^2\mid\FF_k\right]&\leq \left((1-\mu'\alpha_k)^2+\mu'\alpha_k\right)\|x_k-\xstar(y_k)\|^2\\
    &+ \left(2\cc_2\alpha_k^2+2L_0^2\zeta_1\beta_k^2+\frac{L_0^2\zeta_1}{\mu'}\frac{\beta_k^2}{\alpha_k}\right)\left(1+\|x_k\|^2+\|y_k\|^2\right).
\end{align*}
Now for $\ystar\in\Ystar$,  we apply Lemma \ref{lemma:tech}\ref{lemma:xnyn-bound}.
\begin{align*}
    &\EE\left[\|x_{k+1}-\xstar(y_{k+1})\|^2\mid\FF_k\right]\\
    &\leq \left(1-2\mu'\alpha_k+\mu'^2\alpha_k^2+\mu'\alpha_k\right)\|x_k-\xstar(y_k)\|^2\\
    &\;\;+\left(2\cc_2\alpha_k^2+2L_0^2\zeta_1\beta_k^2+\frac{L_0^2\zeta_1}{\mu'}\frac{\beta_k^2}{\alpha_k}\right)\zeta_2\left(1+\|x_k-\xstar(y_k)\|^2+\|y_k-\ystar\|^2\right).
\end{align*}
Now, we use our assumption that $\beta_k^2/\alpha_k^3\leq 1$ (Assumption \ref{assu:stepsize}) to obtain
\begin{align}\label{inter-x}
    \EE\left[\|x_{k+1}-\xstar(y_{k+1})\|^2\mid\FF_k\right]&\leq (1-\mu'\alpha_k)\|x_k-\xstar(y_k)\|^2\nonumber\\
    &\;\;+\Gamma_1\alpha_k^2\left(1+\|x_k-\xstar(y_k)\|^2+\|y_k-\ystar\|^2\right),
\end{align}
where $\Gamma_1\coloneqq \left(\mu'^2+2\cc_2\zeta_2+2L_0^2\zeta_1\zeta_2+2L_0^2\zeta_1\zeta_2/\mu'\right)$. This is our intermediate recursive bound on $\EE[\|x_{k+1}-\xstar(y_{k+1})\|^2\mid\FF_k]$. 

\subsubsection*{Recursive Bound on $\EE[\|y_{k+1}-\ystar\|^2]$}
We first use Corollary 2.14 from \cite{convexbook} which states that for any $\epsilon\in(0,1)$ and $y_1,y_2\in\RR^{d_2}$, we have
$\|(1-\epsilon)y_1+\epsilon y_2\|^2=(1-\epsilon)\|y_1\|^2+\epsilon \|y_2\|^2 -\epsilon(1-\epsilon)\|y_1-y_2\|^2.$ Therefore,
\begin{subequations}\label{split2}
    \begin{align}
    \|y_{k+1}-\ystar\|^2&=\|(1-\beta_k)(y_k-y^*)+\beta_k(g(x_k,y_k)-y^*+M'_{k+1})\|^2\nonumber\\
    &=(1-\beta_k)\|y_k-\ystar\|^2+\beta_k\|g(x_k,y_k)-\ystar+M'_{k+1}\|^2\label{split21}\\
    &\;\;-(1-\beta_k)(\beta_k)\|y_k-g(x_k,y_k)-M'_{k+1}\|^2.\label{split22}
\end{align}
\end{subequations}

\textbf{Term \ref{split21} --- } For the term \eqref{split21}, we first note that
\begin{align*}
    \|g(x_k,y_k)-\ystar+M'_{k+1}\|^2&= \|g(x_k,y_k)-\ystar\|^2+\|M'_{k+1}\|^2+2\langle g(x_k,y_k)-\ystar,M'_{k+1}\rangle.
\end{align*}
The expectation of the third term above is zero when conditioned on $\FF_k$. Now, recall that $h(y)=g(\xstar(y),y)$ is non-expansive. For the first term, we note the following.
\begin{align*}
    &\|g(x_k,y_k)-\ystar\|^2=\|h(y_k)-\ystar+g(x_k,y_k)-h(y_k)\|^2\\
    &= \|h(y_k)-\ystar\|^2+\|g(x_k,y_k)-h(y_k)\|^2+2\langle h(y_k)-\ystar,g(x_k,y_k)-h(y_k)\rangle\\
    &\leq\|h(y_k)-h(\ystar)\|^2+\|g(x_k,y_k)-h(y_k)\|^2+2\|h(y_k)-h(\ystar)\|\|g(x_k,y_k)-h(y_k)\|\\
    &\leq \|y_k-\ystar\|^2+L^2\|x_k-\xstar(y_k)\|^2+2L\|y_k-\ystar\|\|x_k-\xstar(y_k)\|.
\end{align*}
Here, the last inequality follows from the non-expansiveness of map $h(\cdot)$ (Assumption \ref{g-nonexp}), and the Lipschitz nature of map $g(\cdot,\cdot)$ (Assumption \ref{Lipschitz}). For the last term above, we use Young's inequality to note the following. 
$$2L\beta_k\|y_k-\ystar\|\|x_k-\xstar(y_k)\|\leq \frac{\mu'\alpha_k}{2}\|x_k-\xstar(y_k)\|^2+\frac{2L^2\beta_k^2}{\mu'\alpha_k}\|y_k-\ystar\|^2.$$

 Combining these results, we finally get the following result. 
\begin{align*}
    &(1-\beta_k)\EE[\|y_k-\ystar\|^2]+\beta_k\EE[\|g(x_k,y_k)-\ystar+M'_{k+1}\|^2\mid \FF_k]\\
    &\leq\left(L^2\beta_k+\frac{\mu'\alpha_k}{2}\right)\|x_k-\xstar(y_k)\|^2+\! \left(1+\frac{2L^2\beta_k^2}{\mu'\alpha_k}\right)\|y_k-\ystar\|^2+ \beta_k\EE[\|M'_{k+1}\|^2\mid\FF_k].\nonumber
\end{align*}

\textbf{Term \ref{split22} --- }
We first note that 
\begin{align*}
    \|y_k-g(x_k,y_k)-M'_{k+1}\|^2=\|y_k-g(x_k,y_k)\|^2+\|M'_{k+1}\|^2-2\langle y_k-g(x_k,y_k),M'_{k+1}\rangle.
\end{align*}
As $M'_{k+1}$ is a martingale difference sequence, the expectation of the third term is zero.
\begin{align*}
    &-\beta_k(1-\beta_k)\EE[\|y_k-g(x_k,y_k)-M'_{k+1}\|^2\mid\FF_k]\nonumber\\
    &=-\beta_k(1-\beta_k)\|y_k-g(x_k,y_k)\|^2+\left(-\beta_k+\beta_k^2\right)\EE[\|M'_{k+1}\|^2\mid\FF_k].
\end{align*}

Combining the bounds for terms \eqref{split21} and \eqref{split22}, we now have
\begin{align*}
    \EE[\|y_{k+1}-\ystar\|^2\mid\FF_k]&\leq -\beta_k(1-\beta_k)\|y_k-g(x_k,y_k)\|^2+\beta_k^2\EE[\|M'_{k+1}\|^2\mid\FF_k]\\
    &\;\;+ \left(1+\frac{2L^2\beta_k^2}{\mu'\alpha_k}\right)\|y_k-\ystar\|^2+\left(L^2\beta_k+\frac{\mu'\alpha_k}{2}\right)\|x_k-\xstar(y_k)\|^2.
\end{align*}
Now using Assumption \ref{Martingale} and Lemma \ref{lemma:tech}\ref{lemma:xnyn-bound}, we have that 
$\EE[\|M'_{k+1}\|^2\mid\FF_k]$ is bounded by $\cc_2\zeta_2\left(1+\|x_k-\xstar(y_k)\|^2+\|y_k-\ystar\|^2\right).$
Hence we have 
\begin{align*}
    &\EE[\|y_{k+1}-\ystar\|^2\mid\FF_k]\leq \cc_2\zeta_2\beta_k^2-\beta_k(1-\beta_k)\|y_k-g(x_k,y_k)\|^2\\
    &\;\;+ \left(1+\frac{2L^2\beta_k^2}{\mu'\alpha_k}+\cc_2\zeta_2\beta_k^2\right)\|y_k-\ystar\|^2+\left(L^2\beta_k+\frac{\mu'\alpha_k}{2}+\cc_2\zeta_2\beta_k^2\right)\|x_k-\xstar(y_k)\|^2.
\end{align*}
Using Assumption \ref{assu:stepsize}, we have $\beta_k^2\leq \alpha_k^3$. We further have the assumption that $\beta_k/\alpha_k\leq \mu'/(2L^2)$ (and hence $L^2\beta_k\leq \mu'\alpha_k/2$).   
\begin{align}\label{inter-y}
    \EE\left[\|y_{k+1}-\ystar\|^2\mid\FF_k\right]&\leq \left(1+\Gamma_2\alpha_k^2\right)\|y_k-\ystar\|^2+(\mu'\alpha_k+\Gamma_2\alpha_k^2)\|x_k-\xstar(y_k)\|^2\nonumber\\
    &\;\;+\Gamma_2\alpha_k^2-\beta_k(1-\beta_k)\|y_k-g(x_k,y_k)\|^2.
\end{align}
Here $\Gamma_2=2L^2/\mu'+\cc_2\zeta_2$.
This is our recursive bound on $\EE[\|y_{k+1}-\ystar\|^2\mid\FF_k]$. 
\end{proof}

\subsection{Proof for Lemma \ref{lemma:bounded-expectation}}

\begin{proof}[\textbf{Proof for Lemma \ref{lemma:bounded-expectation}}] The lemma has two parts; we first show the stability of iterates, and then show that the iterates are bounded in expectation by a constant.
\subsubsection*{Almost Sure Boundedness}
Adding the two intermediate bounds from Lemma \ref{lemma:inter}, we get
\begin{align}\label{eqn:inter-sum}
    &\EE\left[\|x_{k+1}-\xstar(y_{k+1})\|^2+\|y_{k+1}-\ystar\|^2\mid\FF_k\right]\nonumber\\
    &\leq \left(1+(\Gamma_1+\Gamma_2)\alpha_k^2\right)\left(\|x_k-\xstar(y_k)\|^2+\|y_k-\ystar\|^2\right)+(\Gamma_1+\Gamma_2)\alpha_k^2.
\end{align}
Here we have dropped the term $-\beta_k(1-\beta_k)\|y_k-g(x_k,y_k)\|^2$ as $0<\beta_k<1$. Define $V_k=\left(\|x_k-\xstar(y_k)\|^2+\|y_k-\ystar\|^2\right), A_k=(\Gamma_1+\Gamma_2)\alpha_k^2$. Then, we have $\EE\left[V_{k+1}\mid\FF_k\right]\leq (1+A_k)V_k+A_k.$ Since $\alpha_k$ is square-summable, $A_k$ is summable. Then we can apply Robbins-Siegmund Theorem \cite{Robbins-Siegmund} which says that $V_k$ converges to a finite random variable. In particular, $\sup_k V_k<\infty$ almost surely, and hence using the Lipschitz nature of the map $\xstar(\cdot)$, $\sup_k (\|x_k\|+\|y_k\|)<\infty$.

\subsubsection*{Boundedness in Expectation}
Next, we prove that the iterates are bounded in expectation. Taking expectation on both sides in \eqref{eqn:inter-sum}, we get the following. 
\begin{align*}
    &\EE\left[\|x_{k+1}-\xstar(y_{k+1})\|^2+\|y_{k+1}-\ystar\|^2\right]\\
    &\leq \left(1+(\Gamma_1+\Gamma_2)\alpha_k^2\right)\EE\left[\|x_k-\xstar(y_k)\|^2+\|y_k-\ystar\|^2\right]+(\Gamma_1+\Gamma_2)\alpha_k^2.
\end{align*}
Expanding the recursion, we get the following.
\begin{align*}
    &\EE\left[\|x_k-\xstar(y_k)\|^2+\|y_k-\ystar\|^2\right]\\
    &\leq \prod_{i=0}^{k-1} \left(1+(\Gamma_1+\Gamma_2)\alpha_i^2\right)\EE\left[\|x_0-\xstar(y_0)\|^2+\|y_0-\ystar\|^2\right]\\
    &\;\;+\sum_{i=0}^{k-1}(\Gamma_1+\Gamma_2)\alpha_i^2\prod_{j=i+1}^{k-1}\left(1+(\Gamma_1+\Gamma_2)\alpha_j^2\right)\\
    &\leq e^{\left((\Gamma_1+\Gamma_2)\sum_{i=0}^{\infty}\alpha_i^2\right)}\left(\EE\left[\|x_0-\xstar(y_0)\|^2+\|y_0-\ystar\|^2\right]+\sum_{i=0}^{k-1}(\Gamma_1+\Gamma_2)\alpha_i^2\right).
\end{align*}
Now, we know that the stepsize sequence $\alpha_k^2$ is summable. Let $D_1\coloneqq \sum_{k=0}^{\infty}\alpha_k^2<\infty$. Then we have 
\begin{align*}
    &\EE\left[\|x_k-\xstar(y_k)\|^2+\|y_k-\ystar\|^2\right]\\
    &\leq e^{D_1(\Gamma_1+\Gamma_2)}\left((\Gamma_1+\Gamma_2)D_1+\|x_0-\xstar(y_0)\|^2+\|y_0-\ystar\|^2\right)\\
    &\leq \kappa_1(\kappa_1+\|x_0-\xstar(y_0)\|^2+\|y_0-\ystar\|^2)\eqqcolon \Gamma_3.
\end{align*}
Here $\kappa_1=e^{D_1(\Gamma_1+\Gamma_2)}$. Hence the iterates $x_k$ and $y_k$ are bounded in expectation. 
\end{proof}

\subsection{Proof for Lemma \ref{lemma:wp1_and_bound}}
\begin{proof}[\textbf{Proof for Lemma \ref{lemma:wp1_and_bound}}]
This lemma has three parts. First, we show the almost sure convergence of iterates, and then work on mean square bounds for the iterates.
\subsubsection*{Almost Sure Convergence}
We finally prove almost sure convergence for our iterates. We have already shown the stability of iterates in Lemma \ref{lemma:bounded-expectation}. For a fixed $y$, $f(x,y)$ is contractive in $x$, and hence $\xstar(y)$ is a globally asymptotically stable equilibrium of the ODE limit (Theorem 12.1 from \cite{Borkar-book}) $$\dot{x}(t)=f(x(t),y)-x(t).$$ We have also shown that $\xstar(y)$ is Lipschitz in $y$ (Lemma \ref{lemma:xstar-lip}). And hence using Lemma 8.1 from \cite{Borkar-book}, $\|x_t-\xstar(y_t)\|\rightarrow 0$ with probability $1$. Now consider the ODE limit $$\dot{y}(t)=g(\xstar(y(t)),y(t))-y(t).$$ Theorem 12.2 and Remark 3 (mentioned below the theorem) from \cite{Borkar-book} together imply that $y_t$ almost surely converges to the set $\Ystar$.

\subsubsection*{Bound on $\EE[\|x_k-\xstar(y_k)\|^2]$}
Applying Lemma \ref{lemma:bounded-expectation} on the bound from Lemma \ref{lemma:inter} (after taking expectation), we get 
$$\EE\left[\|x_{k+1}-\xstar(y_{k+1})\|^2\right]\leq (1-\mu'\alpha_k)\EE\left[\|x_k-\xstar(y_k)\|^2\right]+\Gamma_1(1+\Gamma_3)\alpha_k^2.$$
We get the following bound on expanding the above recursion. 
$$\EE\left[\|x_k-\xstar(y_k)\|^2\right]\leq \|x_0-\xstar(y_0)\|^2\prod_{i=0}^{k-1}(1-\mu'\alpha_i)+ \Gamma_1(1+\Gamma_3)\sum_{i=0}^{k-1}\alpha_i^2\prod_{j=i+1}^{k-1}(1-\mu'\alpha_j).$$

For the first term above, there exists $D_2$ such that $1-\mu'\alpha/(k+K_1)^{\afrak}\leq 1-1/(k+1+D_2)$. Using Corollary 2.1.2 from \cite{Zaiwei}, we have 
$$\prod_{i=0}^{k-1}(1-\mu'\alpha_i)\leq \prod_{i=0}^{k-1}\left(1-\frac{1}{i+1+D_2}\right)\leq \frac{D_2}{k+D_2}.$$
For the second term, we use Corollary 2.1.2 from \cite{Zaiwei} to get the following. 
$$\sum_{i=0}^{k-1}\alpha_i^2\prod_{j=i+1}^{k-1}(1-\mu'\alpha_j)\leq \frac{2}{\mu'}\alpha_k.$$
 Here we use the fact that $\alpha_k$ is of the form $\alpha/(k+K_1)^\afrak$ where $0<\afrak<1$ and $K_1\geq \left(\frac{2\afrak}{\mu\alpha}\right)^{1-\afrak}$. Hence we get the following bound.
\begin{equation}\label{final-bound-x}
    \EE\left[\|x_k-\xstar(y_k)\|^2\right]\leq \frac{C_1}{(k+K_1)^{\afrak}}=D_3\alpha_k\leq \frac{C_1}{(k+1)^{\afrak}}.
\end{equation}
Here $C_1=D_3\alpha=\kappa_2\left(1+\|x_0-\xstar(y_0)\|^2+\|y_0-\ystar\|^2\right)$ where $\kappa_2=\max\{\mu'D_2+2\Gamma_1\alpha\kappa_1,2\Gamma_1\alpha(1+\kappa_1^2)\}/\mu'$. This completes our bound on $\EE\left[\|x_k-\xstar(y_k)\|^2\right]$.

\subsubsection*{Bound on summation of $\beta_i\EE[\|h(y_i)-y_i\|^2]$} We first recall that $h(y_k)=g(\xstar(y_k),y_k)$.
Combining \eqref{final-bound-x} and Lemma \ref{lemma:bounded-expectation} with the bound in Lemma \ref{lemma:inter},
\begin{align*}
    &\EE\left[\|y_{k+1}-\ystar\|^2\right]\\
    &\leq \EE\left[\|y_k-\ystar\|^2\right]+(\Gamma_2(1+\Gamma_3)+D_3\mu')\alpha_k^2-\beta_k(1-\beta_k)\EE\left[\|y_k-g(x_k,y_k)\|^2\right]\\
    &\leq \EE\left[\|y_k-\ystar\|^2\right]+(\Gamma_2(1+\Gamma_3)+D_3\mu'+D_3L^2)\alpha_k^2-(\beta_k/4)\EE\left[\|h(y_k)-y_k\|^2\right].
\end{align*}
For the second inequality, we use the assumption that $\beta_k\leq 1/2$, the fact that $\EE[\|h(y_k)-y_k\|^2]\leq 2\EE[\|g(x_k,y_k)-y_k\|^2]+2L^2\EE[\|x_k-\xstar(y_k)\|^2]$, and that $\EE[\|x_k-\xstar(y_k)\|^2]$ is bounded by $D_3\alpha_k$.

Hence, we have the following inequality which holds for all $i\geq 0$.
$$\frac{\beta_i}{4}\EE\left[\|h(y_i)-y_i\|^2\right]\leq \!\EE\left[\|y_i-\ystar\|^2\right]-\EE\left[\|y_{i+1}-\ystar\|^2\right]+(\Gamma_2(1+\Gamma_3)+D_3\mu'+D_3L^2)\alpha_i^2.$$
Summing above inequality from $i=0$ to $k$, we get 
\begin{align*}
    \sum_{i=0}^k\beta_i\EE\left[\|h(y_i)-y_i\|^2\right]&\leq 4\|y_0-\ystar\|^2+4(\Gamma_2(1+\Gamma_3)+D_3\mu'+2D_3L^2)D_1\eqqcolon \Gamma_4,
\end{align*}
where $D_1=\sum\alpha_k^2$ as defined in the proof for Lemma \ref{lemma:bounded-expectation}.
\end{proof}

\subsection{Proof for Lemma \ref{lemma:h-results}}

\begin{proof}[\textbf{Proof for Lemma \ref{lemma:h-results}}]
Define $U_{k+1}=(1-\beta_k)U_k+\beta_kM'_{k+1}$, an averaged error sequence, with $U_0=0$. Define $z_k=y_k-U_k$. We first show that $\EE[\|z_k-y_k\|^2]$ decays sufficiently fast.
\subsubsection*{Bound on $\EE[\|U_k\|^2]$}
Note that $$U_k=\sum_{i=0}^{k-1}\beta_i\prod_{j=i+1}^{k-1}(1-\beta_j)M'_{i+1}.$$
Using the property that the terms of a martingale difference sequence are orthogonal, 
\begin{align}\label{bound-U}
    \EE[\|y_k-z_k\|^2]= \EE\left[\|U_k\|^2\right]&\leq \sum_{i=0}^{k-1} \left(\beta_i\prod_{j=i+1}^{k-1}(1-\beta_j)\right)^2\EE\left[\|M'_{i+1}\|^2\right]\nonumber\\
    &\stackrel{(a)}{\leq} \sum_{i=0}^{k-1} \beta_i^2\prod_{j=i+1}^{k-1}(1-\beta_j) D_4\nonumber\\
    &\stackrel{(b)}{\leq} 2D_4\beta_k\eqqcolon \Gamma_5\beta_k.
\end{align}
Here, inequality (a) follows from Assumption \ref{Martingale} and Lemma \ref{lemma:bounded-expectation}, which together imply that $M'_{i+1}$ is bounded in expectation by some constant $D_4$. Inequality (b) follows from Corollary 2.1.2 from \cite{Zaiwei} and our assumption that $K_1\geq \left(\frac{2\bfrak}{\beta}\right)^{1/1-\bfrak}$.

\subsubsection*{Bound on summation of $\beta_i\EE[\|h(z_i)-z_i\|^2]$}
We next write $\{z_k\}$ as an iteration. We have the iteration $y_{k+1}=(1-\beta_k)y_k+\beta_k(g(x_k,y_k)+M'_{k+1})$. Then using the definition of $z_k$, we get
\begin{align*}
    z_{k+1}+U_{k+1}=(1-\beta_k)(z_k+U_k)+\beta_k g(x_k,y_k)+\beta_kM'_{k+1}. 
\end{align*}
Recall that $h(y)=g(\xstar(y),y)$. This gives us the iteration 
\begin{align*}
    z_{k+1}&=(1-\beta_k)z_k+\beta_kg(x_k,y_k)=z_k+\beta_k(h(z_k)-z_k+e_{k+1}).
\end{align*}
where $e_{k+1}=g(x_k,y_k)-h(y_k)+h(y_k)-h(z_k)$. Next, we write $\|h(z_i)-z_i\|$ in terms of $\|h(y_i)-y_i\|$. For this note the following.
\begin{align*}
    \|h(z_i)-z_i\|&\leq \|h(y_i)-y_i\|+\|h(z_i)-h(y_i)\|+\|y_i-z_i\|\leq\|h(y_i)-y_i\|+2\|y_i-z_i\|.
\end{align*}
Here, the second inequality follows from the non-expansive nature of function $h(\cdot)$. This gives us the inequality
$\|h(z_i)-z_i\|^2\leq 2\|h(y_i)-y_i\|^2+8\|y_i-z_i\|^2.$
Finally, to obtain a summation similar to Lemma \ref{lemma:wp1_and_bound} c), 
\begin{align*}
    \sum_{i=0}^k\beta_i\EE\left[\|h(z_i)-z_i\|^2\right]&\leq 2\sum_{i=0}^k\beta_i\EE\left[\|h(y_i)-y_i\|^2+4\|y_i-z_i\|^2\right]\\
    &\leq 2\Gamma_4+8\sum_{i=0}^{k}\beta_i\EE\left[\|y_i-z_i\|^2\right]\leq 2\Gamma_4+8\Gamma_5\sum_{i=0}^{k}\beta_i^2.
\end{align*}
Note that $\sum_i\beta_i^2$ is also bounded by $D_1$, and hence for $\Gamma_6\coloneqq2\Gamma_4+8\Gamma_5D_1$, we get
\begin{equation}\label{bound_h-z_i-1}
    \sum_{i=0}^k\beta_i\EE\left[\|h(z_i)-z_i\|^2\right]\leq \Gamma_6.
\end{equation}

\subsubsection*{Bound on $\EE[\|h(z_i)-z_i\|^2]$ in terms of $\EE[\|h(z_k)-z_k\|^2]$}
For simplicity, we define $p_i=h(z_i)-z_i$. We next obtain a bound for $\EE[\|p_i\|^2]$ in terms of $\EE[\|p_k\|^2]$ for all $i\leq k$. Recall that $z_{k+1}=z_k+\beta_k(h(z_k)-z_k+e_{k+1})=z_k+\beta_k(p_k+e_{k+1})$. We have the following relation. 
\begin{align*}
    \|p_{i+1}\|^2&=\|p_i\|^2+\|p_{i+1}-p_i\|^2+2\langle p_{i+1}-p_i,p_i\rangle\\
    &=\|p_i\|^2+\|p_{i+1}-p_i\|^2+\frac{2}{\beta_i}\langle p_{i+1}-p_i,z_{i+1}-z_i-\beta_i e_{i+1}\rangle.
\end{align*}
Now,
\begin{align*}
    2\langle p_{i+1}-p_i,z_{i+1}-z_i\rangle&=2\langle h(z_{i+1})-h(z_i)-z_{i+1}+z_i,z_{i+1}-z_i\rangle\\
    &=2\langle h(z_{i+1})-h(z_i),z_{i+1}-z_i\rangle-2\|z_{i+1}-z_i\|^2\\
    &=\|h(z_{i+1})-h(z_i)\|^2+\|z_{i+1}-z_i\|^2-2\|z_{i+1}-z_i\|^2\\
    &\;\;-\|h(z_{i+1})-h(z_i)-z_{i+1}+z_i\|^2\\
    &= \|h(z_{i+1})-h(z_i)\|^2-\|z_{i+1}-z_i\|^2-\|p_{i+1}-p_i\|^2\\
    &\leq -\|p_{i+1}-p_i\|^2.
\end{align*}
Here, the final inequality follows from the non-expansive nature of map $h(\cdot)$. Now,
\begin{align*}
    \|p_{i+1}\|^2&\leq \|p_i\|^2+\|p_{i+1}-p_i\|^2-\frac{2}{\beta_i}\|p_{i+1}-p_i\|^2-2\langle p_{i+1}-p_i,e_{i+1}\rangle\\
    &=\|p_i\|^2-\frac{1-\beta_i}{\beta_i}\|p_{i+1}-p_i\|^2-2\langle p_{i+1}-p_i,e_{i+1}\rangle.
\end{align*}
Now, note that 
$$\left\|p_{i+1}-p_i+\frac{\beta_i}{1-\beta_i}e_{i+1}\right\|^2=\|p_{i+1}-p_i\|^2+\left(\frac{\beta_i}{1-\beta_i}\right)^2\|e_{i+1}\|^2+\frac{2\beta_i}{1-\beta_i}\langle p_{i+1}-p_i,e_{i+1}\rangle.$$
Hence 
\begin{align*}
    \|p_{i+1}\|^2&\leq \|p_i\|^2-\frac{1-\beta_i}{\beta_i}\left\|p_{i+1}-p_i+\frac{\beta_i}{1-\beta_i}e_{i+1}\right\|^2+\frac{\beta_i}{1-\beta_i}\|e_{i+1}\|^2\\
    &\leq \|p_i\|^2+\frac{\beta_i}{1-\beta_i}\|e_{i+1}\|^2\leq \|p_i\|^2+2\beta_i\|e_{i+1}\|^2.
\end{align*}
Here, the final inequality follows from the fact that $\beta_i\leq 0.5$. Recall that $e_{i+1}=g(x_i,y_i)-h(y_i)+h(y_i)-h(z_i)$. Then 
\begin{align*}
    \EE\left[\|e_{i+1}\|^2\right]&\leq 2\EE\left[\|g(x_i,y_i)-h(y_i)\|^2\right]+2\EE\left[\|g(\xstar(y_i),y_i)-g(\xstar(z_i),z_i)\|^2\right]\\
    &\leq 2L^2\EE\left[\|x_i-\xstar(y_i)\|^2\right]+2\EE\left[\|y_i-z_i\|^2\right]\leq 2(L^2D_3+\Gamma_5)\alpha_i.
\end{align*}
Here, the final inequality follows from $\EE[\|x_i-\xstar(y_i)\|^2]\leq D_3\alpha_i$ and $\EE[\|y_i-z_i\|^2]\leq \Gamma_5\beta_i$. We also use the fact that $\beta_i\leq \alpha_i$ for all $i$. Define $\Gamma_7=4(L^2D_3+\Gamma_5)$. Hence
$\EE\left[\|p_{i+1}\|^2\right]\leq \EE\left[\|p_i\|^2\right]+\Gamma_7\beta_i\alpha_i.$
This finally gives us the following bound.
\begin{align*}\EE\left[\|h(z_i)-z_i\|^2\right]\geq \EE\left[\|h(z_k)-z_k\|^2\right]-\sum_{j=i}^{k-1}\Gamma_7\beta_j\alpha_j.\end{align*}
\end{proof}
\subsection{Proof for Theorem \ref{thm:main}}
\begin{proof}[\textbf{Proof for Theorem \ref{thm:main}}] We have already shown the almost sure convergence and the bound for $\EE[\|x_k-\xstar(y_k)\|^2]$ in Lemma \ref{lemma:wp1_and_bound}. We just need to complete the bound for $\EE[\|h(y_i)-y_i\|^2]$ to complete the proof.
Summing above bound and applying bound from \eqref{bound_h-z_i-1}, we get
\begin{equation}\label{bound_h-z_i-2}
    \EE\left[\|h(z_k)-z_k\|^2\right]\sum_{i=0}^k\beta_i\leq \Gamma_6+\Gamma_7\sum_{i=0}^k\beta_i\sum_{j=i}^{k-1}\beta_j\alpha_j.
\end{equation}

To bound the second term, we first note that 
\begin{align*}
    \sum_{j=i}^{k-1}\beta_j\alpha_j&\leq\sum_{j=i}^{\infty}\frac{\alpha\beta}{(j+K_1)^{\afrak+\bfrak}}\leq\sum_{j=i}^{\infty}\frac{\alpha\beta}{(j+1)^{\afrak+\bfrak}}\leq\alpha\beta\left(\frac{1}{(i+1)^{\afrak+\bfrak}}+\int_{x=i+1}^\infty \frac{1}{x^{\afrak+\bfrak}}dx\right)\\
    &\stackrel{(a)}{\leq} \alpha\beta\left(\frac{1}{(i+1)^{\afrak+\bfrak}}+\frac{1}{(\afrak+\bfrak-1)(i+1)^{\afrak+\bfrak-1}}\right)\stackrel{(b)}{\leq} \frac{2\alpha\beta}{\afrak+\bfrak-1}\frac{1}{(i+1)^{\afrak+\bfrak-1}}.
\end{align*}
Here, inequality (a) and (b) both follow from the fact that $1< \afrak+\bfrak<2$, which follows from our assumption that $\alpha_k$ and $\beta_k$ are both non-summable but square summable sequences. Next, 
\begin{align*}
    &\sum_{i=0}^k\beta_i\sum_{j=i}^{k-1}\beta_j\alpha_j\leq \frac{2\alpha\beta^2}{\afrak+\bfrak-1}\sum_{i=0}^{k-1}\frac{1}{(i+1)^{\afrak+2\bfrak-1}}\stackrel{(c)}{\leq} \frac{2\alpha\beta^2}{\afrak+\bfrak-1}\sum_{i=0}^\infty\frac{1}{(i+1)^{4\afrak-1}}\stackrel{(d)}{\eqqcolon} D_5.
\end{align*}
 Inequality (c) follows from the fact that $2\bfrak\geq 3\afrak$, which follows from our assumption that $\beta_k^2\leq \alpha_k^3$. Finally, inequality (d) follows from the fact that $4\afrak-1>1$, which follows from our assumption that $\afrak>0.5$ or $\alpha_k$ is square summable. 

Hence the right hand side of the inequality \eqref{bound_h-z_i-2} is bounded by a constant. Note that $\sum_{i=0}^k \beta_i\geq (k+1)\beta_k,$ which follows from our assumption that the sequence $\beta_k$ is non-increasing. This implies that 
$$\EE\left[\|h(z_k)-z_k\|^2\right]\leq \frac{\Gamma_6+\Gamma_7D_5}{(k+1)\beta_k}.$$ 
To complete our proof, we recall that $\|h(y_k)-y_k\|^2\leq 2\|h(z_k)-z_k\|^2+8\|y_k-z_k\|^2.$
This implies that 
\begin{align}\label{final-bound-y}
   \EE\left[\|h(y_k)-y_k\|^2\right] &\leq \frac{2\Gamma_6+2\Gamma_7D_5}{(k+1)\beta_k}+8\Gamma_5\beta_k\leq \frac{C_2}{(k+1)^{1-\bfrak}},
\end{align}
where $C_2\coloneqq (2\Gamma_6+2\Gamma_7D_5)/\beta+8\Gamma_5\beta$. Here, the first inequality follows from \eqref{bound-U}, and the last inequality follows from the fact that $1/(k+1)^{\bfrak}\leq 1/(k+1)^{1-\bfrak}$, which follows from our assumption that $\bfrak>1/2$. This completes the proof.
\end{proof}

\subsection{A Technical Lemma}\label{app:tech}
\begin{lemma}\label{lemma:tech}
Suppose Assumptions \ref{Lipschitz} and \ref{Martingale} hold. For all $k\geq 0$ and $\ystar\in\Ystar$, the iterates $\{x_k,y_k\}$ generated by iteration \eqref{iter-main} satisfy:
    \begin{enumerate}[label=\alph*)]
        \item $\EE[\|\xstar(y_{k+1})-\xstar(y_k)\|^2\mid\FF_k]\leq L_0^2\zeta_1\beta_k^2(1+\|x_k\|^2+\|y_k\|^2)$ for constant $\zeta_1>0$. \label{lemma:yn-diff}
        \item $1+\|x_k\|^2+\|y_k\|^2\leq \zeta_2(1+\|x_k-\xstar(y_k)\|^2+\|y_k-\ystar\|^2)$ for constant $\zeta_2>0$. \label{lemma:xnyn-bound}
    \end{enumerate}
\end{lemma}
\begin{proof}
a). Note that 
\begin{align*}
   \|\xstar(y_{k+1})-\xstar(y_k)\|^2&\leq L_0^2\|y_{k+1}-y_k\|^2\leq \beta_k^2L_0^2\|g(x_k,y_k)-y_k+M'_{k+1}\|^2.
\end{align*}
Using Assumption \ref{Lipschitz}, $\|g(x_k,y_k)-y_k\|^2\leq 2(1+\cc_1)(1+\|x_k\|^2+\|y_k\|^2)$. Using Assumption \ref{Martingale}, we get 
$\EE[\|M'_{k+1}\|^2\mid\FF_k]\leq \cc_2(1+\|x_k\|^2+\|y_k\|^2)$. Combining these bounds completes the proof with $\zeta_1=4(1+\cc_1)+2\cc_2$.

b). Note that $\|x_k\|^2\leq 3\|x_k-\xstar(y_k)\|^2+3\|\xstar(y_k)-\xstar(\ystar)\|^2+3\|\xstar(\ystar)\|^2$. Similarly, $\|y_k\|^2\leq 2\|y_k-\ystar\|^2+2\|\ystar\|^2$. Hence 
$$1+\|x_k\|^2+\|y_k\|^2\leq (1+3\|\xstar(\ystar)\|^2+2\|\ystar\|^2)+3\|x_k-\xstar(y_k)\|^2+(3L_0^2+2)\|y_k-\ystar\|^2.$$
This completes the proof with $\zeta_2=\max\{(1+3\|\xstar(\ystar)\|^2+2\|\ystar\|^2),3,(3L_0^2+2)\}$.
\end{proof}

\section{Proofs from Section \ref{sec:applications}: Applications}\label{app:applications-proof}
Throughout this appendix, we first show that, when the iterations are written in a root-finding formulation, the operator in the faster time-scale is strongly monotone and the operator in the slower time-scale is co-coercive. We then use the equivalence between formulations from Appendix \ref{app:relations} to show that the fixed-point mappings are contractive and non-expansive, respectively.
\subsection{Proof for Corollary \ref{coro:minimax}}
\begin{proof}[Proof for Corollary \ref{coro:minimax}]
    Since $H(x,y)$ is $\rho$-strongly-concave in $x$, $-\nabla_x H(x,y)$ is $\rho$-strongly monotone \cite[Theorem 2.1.9]{Nesterov}. 
    Define $f(x,y)=x+(\rho/L^2)\nabla_xH(x,y)$. Then $f(x,y)$ is $\mu-$contractive in $x$ for fixed $y$, where $\mu=\sqrt{1-\rho^2/L^2}$.
    
    Recall that $\Phi(y)=\max_{x\in\RR^{d_1}} H(x,y)$. Then Lemma 4.3 from \cite{jordan-ttsgda} shows that $\nabla \Phi(y)=\nabla_y H(\xstar(y),y)$, and $\nabla \Phi(y)$ is $2L_0L$-Lipschitz where $L_0=L/(1-\mu)$. Since $H(x,y)$ is convex in $y$, $\Phi(y)=\max_x H(x,y)$ is also convex in $y$. 
    This implies that $\nabla \Phi(y)$ is $1/(2LL_0)$ co-coercive \cite[Theorem 2.1.5]{Nesterov}.
    Now, let $g(x,y)=y-\frac{1}{LL_0}\nabla_yH(x,y)$, and hence $g(\xstar(y),y)=y-\frac{1}{LL_0}\nabla \Phi(y)$. Hence $g(\xstar(y),y)$ is non-expansive. 
    
    The iterates in \eqref{iter-minimax} can then be written as $$x_{k+1}=x_k+\alpha_k'(f(x_k,y_k)-x_k+\widetilde{M}_{k+1})\;\;\text{and}\;\;y_{k+1}=y_k+\beta_k'(g(x_k,y_k)-y_k+\widetilde{M}'_{k+1}),$$
    where $$\alpha_k'=(L^2/\rho)\alpha_k, \beta_k'=LL_0\beta_k, \widetilde{M}_{k+1}=(\rho/L^2)M_{k+1}\; \text{and} \;\widetilde{M}'_{k+1}=1/(LL_0)M'_{k+1}.$$ Then above iteration satisfies Assumptions \ref{f-contrac}-\ref{assu:stepsize}. Application of Theorem \ref{thm:main} gives us the required result.
\end{proof}
\subsection{Proof for Corollary \ref{coro:linear}}
\begin{proof}[Proof for Corollary \ref{coro:linear}]
The matrix $A_{11}$ is positive definite and hence satisfies $x^TA_{11}x\geq \lambda_{A_{11},min}\|x\|^2$ and $\|A_{11}x\|\leq \lambda_{A_{11},max}\|x\|$. Here $\lambda_{A_{11},min},\lambda_{A_{11},max}>0$ are the smallest and largest eigenvalues of $A_{11}$, respectively. Hence $f(x,y)=x-\frac{\lambda_{A_{11},min}}{\lambda_{A_{11},max}^2}(A_{11}x+A_{12}y-b_1)$ is $\mu$-contractive in $x$, where $\mu=\sqrt{1-\lambda_{A_{11},min}^2/\lambda_{A_{11},max}^2}$. 

Now, $\Delta=A_{22}-A_{21}A_{11}^{-1}A_{12}$ is non-zero PSD, and hence $\Delta$ is $1/\lambda_{\Delta,max}$-co-coercive \cite[Pg. 79]{VI-book}. Then define $g(x,y)=y-(2/\lambda_{\Delta,max})(A_{21}x+A_{22}y-b_2)$, and hence $g(\xstar(y),y)$ is non-expansive. Here $\xstar(y)=A_{11}^{-1}(-A_{12}y+b_1)$.

Finally, the iterates in \eqref{iter-linear} can be rewritten as
$$x_{k+1}=x_k+\alpha_k'(f(x_k,y_k)-x_k+M_{k+1})\;\;\text{and}\;\;y_{k+1}=y_k+\beta_k'(g(x_k,y_k)-y_k+M'_{k+1}),$$
where $\alpha_k'=\alpha_k(\lambda_{A_{11},max}^2/\lambda_{A_{11},min})$ and $\beta_k'=(\lambda_{\Delta,max}/2)\beta_k$. Also, note that 
\begin{align*}
    M_{k+1}&=\frac{\lambda_{A_{11},min}}{\lambda_{A_{11},max}^2}\left(\tilde{b}_1^{(k+1)}-b_1-(\tilde{A}_{11}^{(k+1)}-A_{11})x-(\tilde{A}_{12}^{(k+1)}-A_{12})y\right)\\
    M'_{k+1}&=\frac{2}{\lambda_{\Delta,max}}\left(\tilde{b}_2^{(k+1)}-b_2-(\tilde{A}_{21}^{(k+1)}-A_{21})x-(\tilde{A}_{22}^{(k+1)}-A_{22})y\right)
\end{align*}
satisfy Assumption \ref{Martingale}. Theorem \ref{thm:main} can then be applied to get the required result.
\end{proof}

\subsection{Proof for Lemma \ref{lemma:Lagrangian}}
\begin{proof}[Proof for Lemma \ref{lemma:Lagrangian}]
\textit{a).} Recall that $\xstar(\lambda)$ is the unique solution to $\nabla H_0(x)=A^T\lambda$. Since $H_0(\cdot)$ is strongly concave, we have
\begin{align*} 
    &\rho\|\xstar(\lambda_1)-\xstar(\lambda_2)\|^2\leq -\langle \xstar(\lambda_1)-\xstar(\lambda_2), \nabla H_0(\xstar(\lambda_1))-\nabla H_0(\xstar(\lambda_2))\rangle \\
    &= -\langle \xstar(\lambda_1)-\xstar(\lambda_2), A^T\lambda_1-A^T\lambda_2\rangle= -\langle A\xstar(\lambda_1)-A\xstar(\lambda_2), \lambda_1-\lambda_2\rangle,
\end{align*}
for $\lambda_1,\lambda_2\in\RR^{d_2}$. We have the assumption that $A$ has linearly independent rows, and hence $(AA^T)^{-1}A$ is the pseudoinverse of matrix $A^T$. Hence we have $$(AA^T)^{-1}A(\nabla H_0(\xstar(\lambda_1))-\nabla H_0(\xstar(\lambda_2)))=\lambda_1-\lambda_2.$$  This implies that 
\begin{align*}
\|\lambda_1-\lambda_2\|&\leq \|(AA^T)^{-1}A\|\|\nabla H_0(\xstar(\lambda_1))-\nabla H_0(\xstar(\lambda_2))\|\\
&\leq L\|(AA^T)^{-1}A\|\|\xstar(\lambda_1)-\xstar(\lambda_2)\|.    
\end{align*}
Hence $\langle -A\xstar(\lambda_1)+A\xstar(\lambda_2),\lambda_1-\lambda_1\rangle \geq \rho/(L\|(AA^T)^{-1}A\|)^2\|\lambda_1-\lambda_2\|^2$. This implies that $-A\xstar(\lambda)$ is strongly monotone in $\lambda$.

\textit{b).} Recall that $\xhat(\cdot)$ is the unique solution to $\PP_\XX(x+\nabla H_0(x)-A^T\lambda)=x$. Using \cite[Prop. 1.4.2]{VI-book}, $\xhat(\lambda)$ satisfies:
$\left\langle x-\xhat(\lambda),\nabla H_0(\xhat(\lambda))-A^T\lambda\right\rangle \leq0,\forall x\in\XX,\lambda\in\RR^{d_2}$. Then we have 
\begin{align*}
    \left\langle \xhat(\lambda_1)-\xhat(\lambda_2),\nabla H_0(\xhat(\lambda_2))-A^T\lambda_2\right\rangle &\leq0\\
    \left\langle \xhat(\lambda_2)-\xhat(\lambda_1),\nabla H_0(\xhat(\lambda_1))-A^T\lambda_1\right\rangle &\leq 0
\end{align*}
Combining these equations gives us
$$-\left\langle \xhat(\lambda_1)-\xhat(\lambda_2),\nabla H_0(\xhat(\lambda_1))-\nabla H_0(\xhat(\lambda_2))\right\rangle \leq -\left\langle \xhat(\lambda_1)-\xhat(\lambda_2),A^T\lambda_1-A^T\lambda_2\right\rangle.$$
Then, we have the following.
\begin{align*} 
    &\rho\|\xhat(\lambda_1)-\xhat(\lambda_2)\|^2\leq -\langle \xhat(\lambda_1)-\xhat(\lambda_2), \nabla H_0(\xhat(\lambda_1))-\nabla H_0(\xhat(\lambda_2))\rangle \\
    &\leq -\langle \xhat(\lambda_1)-\xhat(\lambda_2), A^T\lambda_1-A^T\lambda_2\rangle= -\langle A\xhat(\lambda_1)-A\xhat(\lambda_2), \lambda_1-\lambda_2\rangle,
\end{align*}
Now $\|A\xhat(\lambda_1)-A\xhat(\lambda_2)\|^2\leq \|A\|^2\|\xhat(\lambda_1)-\xhat(\lambda_2)\|^2$, and hence $-A\xhat(\lambda)$ is $(\rho/\|A\|^2)$-co-coercive.
\end{proof}

\subsection{Proof for Corollary \ref{coro:Lagrangian}}
\begin{proof}[Proof for Corollary \ref{coro:Lagrangian}]
The iterates in \eqref{iter-Lagrangian} can be rewritten as
$$x_{k+1}=\PP_\XX\big(x_k+\alpha_k'(f(x_k,\lambda_k)-x_k+M_{k+1})\big)\;\;\text{and}\;\;\lambda_{k+1}=\lambda_k+\beta_k'(g(x_k,\lambda_k)-\lambda_k),$$
where $f(x,\lambda)=(\rho/L^2)(\nabla_x H_0(x)-A^T\lambda)+x,\; g(x,\lambda)=(2\rho/\|A\|^2)(-Ax+b_0)+\lambda$ and $\alpha_k'=(L^2/\rho)\alpha_k, \; \beta_k'=(\|A\|^2/2\rho)\beta_k$. Then using Lemma \ref{lemma:Lagrangian}, the above iteration satisfies assumption \ref{g-nonexp-proj}. Assumption \ref{assu:Lagrangian} b) specifies the Slater's condition, implying existence of $\lambda$ such that $A\xhat(\lambda)=b_0$ \cite[Proposition 27.21]{convexbook}. Application of Theorem \ref{thm:main-projected} gives us bounds on $\EE[\|g(x_k,\lambda_k)-\lambda_k\|^2]$ which is equal to $(4\rho^2/\|A\|^4)\EE[\|Ax-b_0\|^2]$. This completes our proof for Corollary \ref{coro:Lagrangian}.
\end{proof}

\section{Proof for Theorem \ref{thm:main-projected}}\label{app:projected-proof}
\begin{proof}[\textbf{Proof for Theorem \ref{thm:main-projected}}] 
We define the function $f_p(x,y)=\PP_\XX(f(x,y))$ as the projection of function $f(x,y)$ onto the set $\XX$. Note that $f_p(x,y)$ is also $\mu-$contractive in $x$, i.e., $\|f_p(x_1,y)-f_p(x_2,y)\|\leq \mu\|x_1-x_2\|$. Similarly, it is $L-$Lipschitz in $y$. This follows from the non-expansive nature of the projection operator. Also, note that $\xhat(y)$ has been defined as the fixed point for $f_p(\cdot,y)$. 

The outline for this proof is the same as the proof for Theorem \ref{thm:main}. 
\subsubsection*{Lipschitz nature of $\xhat(\cdot)$}
The map $\xhat(\cdot)$ is also Lipschitz with parameter $L_0=L/(1-\mu)$. This follows directly from the proof for Lemma \ref{lemma:xstar-lip}, replacing $f$ with $f_p$.
\subsubsection*{Recursive Bound on $\EE[\|x_k-\xhat(y_k)\|^2]$}
    Note that
        \begin{align*}
            \|x_{k+1}-\xhat(y_{k+1})\|^2&=\left\|\PP_\XX\big(x_k+\alpha_k(f(x_k,y_k)-x_k+M_{k+1})\big)-\xhat(y_{k+1})\right\|^2\nonumber\\
            &\stackrel{(a)}{=} \left\|\PP_\XX\big(x_k+\alpha_k(f(x_k,y_k)-x_k+M_{k+1})\big)-\PP_\XX(\xhat(y_{k+1}))\right\|^2\nonumber\\
            &\leq \left\|x_k+\alpha_k(f(x_k,y_k)-x_k+M_{k+1})-\xhat(y_{k+1})\right\|^2\nonumber.
        \end{align*}
Here, equality (a) follows from the fact that $\xhat(y_{k+1})\in\XX$, and the inequality follows from non-expansiveness of the projection operator. Similar to the expansion in \eqref{split1}, 
    \begin{subequations}\label{split1-pro}
        \begin{align}
            &\|x_{k+1}-\xhat(y_{k+1})\|^2=\|x_k-\xhat(y_k)+\alpha_k(f(x_k,y_k)-x_k)\|^2\label{split11-pro}\\
            &\;\;+\|\alpha_kM_{k+1}+\xhat(y_k)-\xhat(y_{k+1})\|^2\label{split12-pro}\\
            &\;\;+2\Big\langle x_k-\xhat(y_k)+\alpha_k(f(x_k,y_k)-x_k),\alpha_kM_{k+1}+\xhat(y_k)-\xhat(y_{k+1})\Big\rangle. \label{split13-pro}
        \end{align}
    \end{subequations}
    For the first term \eqref{split11-pro}, note that
    \begin{align*}
        \|x_k-\xhat(y_k)+\alpha_k(f(x_k,y_k)-x_k)\|^2&\leq \|x_k-\xhat(y_k)\|^2+\alpha_k^2\|f(x_k,y_k)-x_k\|^2\\
        &\;\;+2\alpha_k\langle x_k-\xhat(y_k),f(x_k,y_k)-x_k\rangle.
    \end{align*}
    Using Assumption \ref{Lipschitz}, we get $\|f(x_k,y_k)-x_k\|^2\leq 2(1+\cc_1)(1+\|x_k\|^2+\|y_k\|^2).$ For the third term, we 
    first note the property of a projection that $\langle z-\PP_\XX(z),x-\PP_\XX(z) \rangle\leq 0$ for all $z\in\RR^{d_1}, x\in\XX$. Hence we get 
    \begin{align*}
        \langle f(\xhat(y_k),y_k)-\xhat(y_k),x_k-\xhat(y_k)\rangle\leq 0,
    \end{align*}
    which follows from the fact that $\xhat(y_k)=\PP_\XX(f(\xhat(y_k),y_k))$. This gives us 
    \begin{align*}
        \langle x_k-\xhat(y_k),f(x_k,y_k)-x_k\rangle&\leq \Big\langle x_k-\xhat(y_k),f(x_k,y_k)-x_k-(f(\xhat(y_k),y_k)-\xhat(y_k))\Big\rangle\\
        &\leq -\mu'\|x_k-\xhat(y_k)\|^2.
    \end{align*}
    This follows from the fact that $(x-f(x,y))$ is a $\mu'$-strongly monotone function in $x$ for all $y$. This finally gives us
    \begin{align}\label{split11-pro-simp}
        &\|x_k-\xhat(y_k)+\alpha_k(f(x_k,y_k)-x_k)\|^2\nonumber\\
        &\leq (1-2\mu'\alpha_k)\|x_k-\xhat(y_k)\|^2+2(1+\cc_1)\alpha_k^2(1+\|x_k\|^2+\|y_k\|^2).
    \end{align}
    The simplification for \eqref{split12-pro} is the same as \eqref{split12}, and hence we do not repeat that here. For \eqref{split13}, 
    \begin{align*}
        &2\Big\langle x_k-\xhat(y_k)+\alpha_k(f(x_k,y_k)-x_k),\xhat(y_k)-\xhat(y_{k+1})\Big\rangle\\
        &\leq \mu'\alpha_k\|x_k-\xhat(y_k)+\alpha_k(f(x_k,y_k)-x_k)\|^2+\frac{1}{\mu'\alpha_k}\|\xhat(y_k)-\xhat(y_{k+1})\|^2.
    \end{align*}
    The second term is dealt with as before. For the first term, we use \eqref{split11-pro-simp} as follows \begin{align*}
        &\mu'\alpha_k\|x_k-\xhat(y_k)+\alpha_k(f(x_k,y_k)-x_k)\|^2\\
        &\leq \mu'\alpha_k\|x_k-\xhat(y_k)\|^2+2(1+\cc_1)\alpha_k^2(1+\|x_k\|^2+\|y_k\|^2).
    \end{align*}
    Combining the above results with \eqref{split12-simp} and \eqref{split13-simp}, we get the following bound.
    \begin{align*}
    \EE\left[\|x_{k+1}-\xhat(y_{k+1})\|^2\right]&\leq \left((1-2\mu'\alpha_k+\mu'\alpha_k\right)\EE\left[\|x_k-\xhat(y_k)\|^2\right]\\
    &\!\!\!\!\!\!\!\!\!\!\!\!\!\!\!\!\!\!\!\!\!\!+ \left(4(1+\cc_1)+2\cc_2\alpha_k^2+2L_0^2\zeta_1\beta_k^2+\frac{L_0^2\zeta_1}{\mu'}\frac{\beta_k^2}{\alpha_k}\right)\EE\left[1+\|x_k\|^2+\|y_k\|^2\right].
\end{align*}
Following the same steps as proof for Lemma \ref{lemma:inter}, we get the following bound for $\yhat\in\Yhat$.
Then we have
\begin{align}\label{inter-x-pro}
    \EE\left[\|x_{k+1}-\xhat(y_{k+1})\|^2\right]&\leq (1-\mu'\alpha_k)\EE\left[\|x_k-\xhat(y_k)\|^2\right]\nonumber\\
    &\;\;+\Gamma_{1,p}\alpha_k^2\left(1+\EE\left[\|x_k-\xhat(y_k)\|^2\right]+\EE\left[\|y_k-\yhat\|^2\right]\right),
\end{align}
where $\Gamma_{1,p}\coloneqq \left(4(1+\cc_1)+2\cc_2\zeta_2+2L_0^2\zeta_1\zeta_2+2L_0^2\zeta_1\zeta_2/\mu'\right)$. This is our intermediate recursive bound on $\EE[\|x_{k+1}-\xhat(y_{k+1})\|^2]$. 

\subsubsection*{Recursive Bound on $\EE[\|y_k-\yhat\|^2]$}
This bound follows the exact same proof as Lemma \ref{lemma:inter}, with $\xstar(y)$ replaced with $\xhat(y)$, $h(y)$ replaced with $h_p(y)=g(\xhat(y),y)$ and $\ystar$ replaced with $\yhat$.
\begin{align}\label{inter-y-pro}
    \EE\left[\|y_{k+1}-\yhat\|^2\right]&\leq \left(1+\Gamma_2\alpha_k^2\right)\EE\left[\|y_k-\yhat\|^2\right]+(\mu'\alpha_k+\Gamma_2\alpha_k^2)\EE\left[\|x_k-\xhat(y_k)\|^2\right]\nonumber\\
    &\;\;+\Gamma_2\alpha_k^2-\beta_k(1-\beta_k)\EE\left[\|y_k-g(x_k,y_k)\|^2\right].
\end{align}
Here $\Gamma_2=2L^2/\mu'+\cc_2\zeta_2$.
\subsubsection*{Boundedness in Expectation}
Adding \eqref{inter-x-pro} and \eqref{inter-y-pro} as done in Lemma \ref{lemma:bounded-expectation}, we get that the iterates are bounded in expectation. For $\kappa_{1,p}=e^{D_1(\Gamma_{1,p}+\Gamma_2)}$,
\begin{align*}
    &\EE\left[\|x_k-\xhat(y_k)\|^2+\|y_k-\yhat\|^2\right]\leq \kappa_{1,p}(\kappa_{1,p}+\|x_0-\xhat(y_0)\|^2+\|y_0-\yhat\|^2)\eqqcolon \Gamma_{3,p}.
\end{align*}
\subsubsection*{Almost Sure Convergence}
Next, we prove almost sure convergence for our iterates. The proof to show almost sure stability of iterates is identical to the proof for Lemma \ref{lemma:bounded-expectation}. We have also shown that $\xhat(y)$ is Lipschitz in $y$. For a fixed $y$, the ODE limit of the iteration is $\dot{x}(t)=\Pi_\XX\left(x(t), f(x(t),y)-x\right)$, where $\Pi_\XX(x,v)=\lim_{\delta\rightarrow0}\frac{\PP_\XX(x+\delta v)-x}{\delta}$. Now, we also note that $\xhat(y)$ is the unique solution to the variational inequality VI$(\XX,x-f(x,y))$, given by the solution $x\in\XX$ which satisfies $\langle z-x, x-f(x,y)\rangle\geq 0$ for all $z\in\XX$. This follows from the fact that $x-f(x,y)$ is strongly monotone in $x$ for all $y$, and Proposition 1.5.8 and Theorem 2.3.3 from \cite{VI-book}. Then Theorem 3.7 from \cite{projected_book} states that the solutions of the ODE above converge to the solution of the VI, which is $\xhat(y)$ in this case. Moreover this equilibrium is global exponentially stable. And hence using Lemma 8.1 from \cite{Borkar-book}, $x_t\rightarrow \xhat(y_t)$ with probability $1$. Now consider the ODE limit $\dot{y}(t)=g(\xhat(y(t)),y(t))-y(t).$ Theorem 12.2 and Remark 3 (mentioned below the theorem) from \cite{Borkar-book} together imply that $y_t$ almost surely converges to the set $\Yhat$.
\subsubsection*{Finite Time Bounds}
The rest of the proof for the finite time bounds is the same as the proof for Theorem \ref{thm:main}, with the only difference in the constants ($\Gamma_1$ and $\Gamma_3$ are replaced by $\Gamma_{1,p}$ and $\Gamma_{3,p}$, respectively).
\end{proof}

\section{Proof for Theorem \ref{thm:gradient}}\label{app:gradient-proof}
\begin{proof}[\textbf{Proof for Theorem \ref{thm:gradient}}]
For this proof, we first give intermediate bounds on $\EE[\|x_{k+1}-\xstar(y_{k+1})\|^2]$ and $\EE[J(y_k)]$. For simplicity, define $G(x_k,y_k)=g(x_k,y_k)-y_k$. 
\subsubsection*{Intermediate Bound on $\EE[\|x_{k+1}-\xstar(y_{k+1})\|^2]$}
    We first give a recursive bound on $\EE[\|x_{k+1}-\xstar(y_{k+1})\|^2]$. Similar to \eqref{split1}, we have the following. 
    \begin{subequations}\label{split3}
        \begin{align}
            &\|x_{k+1}-\xstar(y_{k+1})\|^2\nonumber\\
            &=\|x_k-\xstar(y_k)+\alpha_k(f(x_k,y_k)-x_k)\|^2\label{split31}\\
            &\;\;+\|\alpha_kM_{k+1}+\xstar(y_k)-\xstar(y_{k+1})\|^2\label{split32}\\
            &\;\;+2\Big\langle x_k-\xstar(y_k)+\alpha_k(f(x_k,y_k)-x_k),\alpha_kM_{k+1}+\xstar(y_k)-\xstar(y_{k+1})\Big\rangle. \label{split33}
        \end{align}
    \end{subequations}
 The term \eqref{split31} is bounded by $(1-\mu'\alpha_k)^2\|x_k-\xstar(y_k)\|^2$, where $\mu'=1-\mu$ \eqref{split11-simp}.  We deal with the terms \eqref{split32} and \eqref{split33} separately.
        
        \textbf{Term \ref{split32} --- } Using Lemma \ref{lemma:xstar-lip}, we simplify it as follows.
        \begin{align*}
            \|\alpha_kM_{k+1}+\xstar(y_k)-\xstar(y_{k+1})\|^2&\leq 2\alpha_k^2\|M_{k+1}\|^2+2L_0^2\|y_{k+1}-y_k\|^2.
        \end{align*}
        Now, $y_{k+1}-y_k=\beta_k(G(x_k,y_k)-G(\xstar(y_k),y_k))-\beta_k(\gradJ(y_k))+\beta_kM'_{k+1}$.
        Hence, 
        \begin{align}\label{y-k-diff-grad}
            &\|y_{k+1}-y_k\|^2\leq 3L^2\beta_k^2\|x_k-\xstar(y_k)\|^2+3\beta_k^2\|\gradJ(y_k)\|^2+3\beta_k^2\|M'_{k+1}\|^2.
        \end{align}
        Using Assumption \ref{assu:grad-all}c, we get
        \begin{align*}
            &\EE[\|\alpha_kM_{k+1}+\xstar(y_k)-\xstar(y_{k+1})\|^2\mid\FF_k]\nonumber\\
            &\leq 2\cc_2(\alpha_k^2+3L_0^2\beta_k^2)+6L_0^2\beta_k^2\|\gradJ(y_k)\|^2+6L_0^2L^2\beta_k^2\|x_k-\xstar(y_k)\|^2.
        \end{align*}
        
        \textbf{Term \ref{split33} --- } To simplify \eqref{split33}, note that
        \begin{align*}
            &2\Big\langle x_k-\xstar(y_k)+\alpha_k(f(x_k,y_k)-x_k),\xstar(y_k)-\xstar(y_{k+1})\Big\rangle\\
            &\leq \frac{\mu'\alpha_k}{2}\|x_k-\xstar(y_k)+\alpha_k(f(x_k,y_k)-x_k)\|^2+\frac{2}{\mu'\alpha_k}\|\xstar(y_k)-\xstar(y_{k+1})\|^2.
        \end{align*}
        Using \eqref{y-k-diff-grad} and the fact that $M_{k+1}$ is a martingale difference sequence, we get
        \begin{align*}
            &\EE\Big[2\langle x_k-\xstar(y_k)+\alpha_k(f(x_k,y_k)-x_k),\alpha_kM_{k+1}+\xstar(y_k)-\xstar(y_{k+1})\big\rangle\mid\FF_k\Big]\nonumber\\
            &\leq \frac{\mu'\alpha_k}{2}\|x_k-\xstar(y_k)\|^2+\frac{6L_0^2}{\mu'}\frac{\beta_k^2}{\alpha_k}(\cc_2+\|\gradJ(y_k)\|^2+L^2\|x_k-\xstar(y_k)\|^2).
        \end{align*}

Having obtained a bound on the conditional expectation of the three terms in \eqref{split3}, we now combine the bounds after taking expectation. We additionally use our assumption that $\beta_k^2\leq \alpha_k^3\leq \alpha_k^2$.
\begin{align*}
    &\EE\left[\|x_{k+1}-\xstar(y_{k+1})\|^2\right]\leq  D_6\alpha_k^2\EE\left[\|\gradJ(y_k)\|^2\right]+D_7\alpha_k^2\\
    &\;\;+ \left(1-2\mu'\alpha_k+\mu'^2\alpha_k^2+\frac{\mu'\alpha_k}{2}+6L_0^2L^2\alpha_k^2+\frac{6L_0^2L^2}{\mu'}\alpha_k^2\right)\EE\left[\|x_k-\xstar(y_k)\|^2\right].
\end{align*}
Here $D_6=6L_0^2(1+1/\mu')$ and $D_7=2\cc_2(1+3L_0^2+3L_0^2/\mu')$. For appropriate choice of $C_4$ in our assumption that $\alpha\leq C_4$, we get  $(\mu'^2+6L_0^2L^2+6L_0^2L^2/\mu')\alpha_k^2\leq \mu'\alpha_k/2$. 
\begin{align*}
    &\EE\left[\|x_{k+1}-\xstar(y_{k+1})\|^2\right]\\
    &\leq (1-\mu'\alpha_k)\EE\left[\|x_k-\xstar(y_k)\|^2\right]+ D_6\alpha_k^2\EE\left[\|\gradJ(y_k)\|^2\right]+D_7\alpha_k^2.
\end{align*}
This gives us
\begin{align*}
    &\mu'\alpha_k\EE\left[\|x_k-\xstar(y_k)\|^2\right]\leq \EE\left[\|x_k-\xstar(y_k)\|^2\right]-\EE\left[\|x_{k+1}-\xstar(y_{k+1})\|^2\right] \\
    &\;\;\;\;\;\;\;\;\;\;\;\;\;\;\;+ D_6\alpha_k^2\EE\left[\|\gradJ(y_k)\|^2\right]+D_7\alpha_k^2.
\end{align*}
Our intermediate bound for $\EE[\|x_k-\xstar(y_k)\|^2]$ is the following.
\begin{align}\label{inter-x-gradient}
    &\EE\left[\|x_k-\xstar(y_k)\|^2\right]\leq \frac{1}{\mu'\alpha_k}\EE\left[\|x_k-\xstar(y_k)\|^2\right]-\frac{1}{\mu'\alpha_k}\EE\left[\|x_{k+1}-\xstar(y_{k+1})\|^2\right] \nonumber\\
    &\;\;\;\;\;\;\;\;\;\;\;\;\;\;\;+ (D_6/\mu')\alpha_k\EE\left[\|\gradJ(y_k)\|^2\right]+(D_7/\mu')\alpha_k.
\end{align}

\subsubsection*{Recursive Bound on $\EE[J(y_k)]$}
Note that $\gradJ(\cdot)$ is Lipschitz with parameter $L_J\coloneqq L(L_0+1)$. To see this, note that 
\begin{align*}
    \|\gradJ(y_2)-\gradJ(y_1)\|&\leq \|G(\xstar(y_2),y_2)-G(\xstar(y_1),y_1)\|\\
    &\leq L(\|\xstar(y_2)-\xstar(y_1)\|+\|y_2-y_1\|)\leq L(L_0+1)\|y_2-y_1\|.
\end{align*}
Then, using the definition of smoothness, we have the following.
$J(y_2)\leq J(y_1)+\langle\gradJ(y_1),y_2-y_1\rangle+\frac{L_J}{2}\|y_2-y_1\|^2.$
Substituting $y_2=y_{k+1}$ and $y_1=y_k$,
\begin{align}\label{J-split}
    J(y_{k+1})\leq J(y_k)+\langle \gradJ(y_k),y_{k+1}-y_k\rangle+\frac{L_J}{2}\|y_{k+1}-y_k\|^2.
\end{align}
We have already bounded the last term in \eqref{y-k-diff-grad}. We simplify the first term as follows. 

\textbf{Term $\langle \gradJ(y_k),y_{k+1}-y_k\rangle$ --- } 
Using the simplification above \eqref{y-k-diff-grad}, we get 
\begin{align*}
    &\langle \gradJ(y_k),y_{k+1}-y_k\rangle\\
    &=\beta_k \langle \gradJ(y_k),G(x_k,y_k)-G(\xstar(y_k),y_k)\rangle -\beta_k\|\gradJ(y_k)\|^2+\beta_k\langle \gradJ(y_k),M'_{k+1}\rangle.
\end{align*}
Since $M'_{k+1}$ is a martingale difference sequence, we have $\EE[\langle \gradJ(y_k),M'_{k+1}\rangle\mid\FF_k]=0$. For the first term, we use Cauchy-Schwarz inequality, Young's inequality, and the Lipschitzness of function $G(\cdot)$ to obtain -
\begin{align*}
    \beta_k \langle \gradJ(y_k),G(x_k,y_k)-G(\xstar(y_k),y_k)\rangle\leq \beta_k\frac{1}{4}\|\gradJ(y_k)\|^2+\beta_kL^2\|x_k-\xstar(y_k)\|^2.
\end{align*}
Hence, we get the following.
\begin{align*}
    \EE\left[\langle \gradJ(y_k),y_{k+1}-y_k\rangle\mid\FF_k\right]=L^2\beta_k\|x_k-\xstar(y_k)\|^2-\frac{3}{4}\beta_k\|\gradJ(y_k)\|^2.
\end{align*}

We now return to \eqref{J-split}. Using our assumption that $\beta_k^2\leq \beta_k$, and taking expectation, we get our intermediate bound on $\EE[J(y_k)]$.
\begin{align}
    \EE[J(y_{k+1})]&\leq \EE[J(y_k)]+\beta_k(L^2+1.5L_JL^2)\EE\left[\|x_k-\xstar(y_k)\|^2\right]\nonumber\\
    &\;\;+(-0.75\beta_k+1.5L_J\beta_k^2)\EE\left[\|\gradJ(y_k)\|^2\right]+1.5L_J\cc_2\beta_k^2.
\end{align}
\subsection*{Final Bounds}
Having obtained intermediate bounds, we substitute the bound for $\EE[\|x_k-\xstar(y_k)\|^2]$ into our bound for $\EE[J(y_{k+1})]$. This gives us
\begin{align*}
    \EE[J(y_{k+1})]&\leq \EE[J(y_k)]+D_8(\beta_k/\alpha_k)\left(\left[\|x_k-\xstar(y_k)\|^2\right]-\EE\left[\|x_{k+1}-\xstar(y_{k+1})\|^2\right]\right)\\
    &+D_6D_8\beta_k\alpha_k\EE\left[\|\gradJ(y_k)\|^2\right]+D_7D_8\beta_k\alpha_k\\
    &+(-0.75\beta_k+1.5L_J\beta_k^2)\EE\left[\|\gradJ(y_k)\|^2\right]+1.5L_J\cc_2\beta_k^2.
\end{align*}
Here, $D_8=(L^2+1.5L_JL^2)/\mu'$. For appropriate choice of $C_4$ in our assumption that $\alpha\leq C_4$, we get $\beta_k^2\leq \beta_k\alpha_k$, and $(D_6D_8+1.5L_J)\beta_k\alpha_k\leq 0.5\beta_k$.
\begin{align*}
    \EE[J(y_{k+1})]&\leq \EE[J(y_k)]-0.25\beta_k\EE\left[\|\gradJ(y_k)\|^2\right]+(D_7D_8+1.5L_J\cc_2)\beta_k\alpha_k\\
    &\;\;+D_8(\beta_k/\alpha_k)\left(\EE\left[\|x_k-\xstar(y_k)\|^2\right]-\EE\left[\|x_{k+1}-\xstar(y_{k+1})\|^2\right]\right).
\end{align*}
This gives us the following relation for all $i\geq 0$.
\begin{align*}
    \beta_i\EE\left[\|\gradJ(y_i)\|^2\right]&\leq 4\EE[J(y_i)]-4\EE[J(y_{i+1})]+4(D_7D_8+1.5L_J\cc_2)\beta_i\alpha_i\\
    &\;\;+ 4D_8(\beta_i/\alpha_i)\left(\EE\left[\|x_i-\xstar(y_i)\|^2\right]-\EE\left[\|x_{i+1}-\xstar(y_{i+1})\|^2\right]\right).
\end{align*}
Summing over $i=0$ to $k$ gives us
\begin{align}\label{J-almost-done}
    \sum_{i=0}^k\beta_i\EE\left[\|\gradJ(y_i)\|^2\right]&\leq 4J(y_0)-4\EE[J(y_{k+1})]+4(D_7D_8+1.5L_J\cc_2)\sum_{i=0}^k \beta_i\alpha_i\nonumber\\
    &\;\;+4D_8\sum_{i=0}^k \frac{\beta_i}{\alpha_i}\left(\EE\left[\|x_i-\xstar(y_i)\|^2\right]-\EE\left[\|x_{i+1}-\xstar(y_{i+1})\|^2\right]\right).
\end{align}
 We first need to bound the summation on the right-hand side to complete our bound.
\begin{align*}
    &\sum_{i=0}^k \frac{\beta_i}{\alpha_i}\left(\EE\left[\|x_i-\xstar(y_i)\|^2\right]-\EE\left[\|x_{i+1}-\xstar(y_{i+1})\|^2\right]\right)\\
    &\leq (\beta/\alpha)\|x_0^2-\xstar(y_0)\|^2-(\beta_{k}/\alpha_{k})\EE\left[\|x_{k+1}-\xstar(y_{k+1})\|^2\right]\\
    &\;\;-\sum_{i=1}^k \EE[\|x_i-\xstar(y_i)\|^2]\left(\frac{\beta_{i-1}}{\alpha_{i-1}}-\frac{\beta_i}{\alpha_i}\right)\leq  (\beta/\alpha)\|x_0^2-\xstar(y_0)\|^2.
\end{align*}
Here, the final inequality follows from the fact that $\beta_k/\alpha_k$ is a decreasing sequence. Returning to \eqref{J-almost-done}, we have the assumption that the sequence $\alpha_k\beta_k$ is $\mathcal{O}(1/(k+1))$.  We also have the assumption that the function $J(\cdot)$ is lower bounded. Hence the right-hand side of \eqref{J-almost-done} is bounded by $D_9(1+\log(k+1))$ for some constant $D_9>0$. Hence,
\begin{align*}
    D_9(1+\log(k+1))\geq \sum_{i=0}^k\beta_i \EE\left[\|\gradJ(y_i)\|^2\right]&\geq \left(\min_{0\leq i\leq k}\EE\left[\|\gradJ(y_i)\|^2\right] \right)\sum_{i=0}^k\beta_i\\
    &\geq \left(\min_{0\leq i\leq k}\EE\left[\|\gradJ(y_i)\|^2\right] \right)k\beta_k.
\end{align*}
This completes our proof with $$\min_{0\leq i\leq k}\EE\left[\|\gradJ(y_i)\|^2\right]\leq \frac{C_3(1+\log(k+1))}{(k+1)^{1-\bfrak}},$$
where $C_3=D_9/\beta$. The required $C_4$ in the assumption $\alpha\leq C_4$ for the result to hold is $C_4=\min\{1, (1-\mu)^4/(2+24L^4), (1-\mu)^5/(18(L^6+L^5+L^4+L^2+L))\}$.

\end{proof}

\section{Equivalence Between Formulations}\label{app:relations}
We first show how a formulation using contractive and non-expansive maps implies the formulation using strongly monotone and co-coercive operators.
\begin{lemma}\label{lemma:direction-1}
    Consider function $h:\RR^d\mapsto\RR^d$. Then,
    \begin{enumerate}
        \item If $h(\cdot)$ is $\mu$-contractive for $\mu\in[0,1)$, then function $q(x)=x-h(x)$ is $(1-\mu)$-strongly monotone and $(1+\mu)$-Lipschitz.
        \item If $h(\cdot)$ is non-expansive, then function $q(x)=x-h(x)$ is $0.5$-co-coercive.
    \end{enumerate}
\end{lemma}
\begin{proof}
    \textit{a). } For $x_1,x_2\in\RR^d$, 
    \begin{align*}
        \langle x_1-x_2, q(x_1)-q(x_2)\rangle = \|x_1-x_2\|^2-\langle x_1-x_2,h(x_1)-h(x_2)\rangle.
    \end{align*}
    Now, $\langle x_1-x_2,h(x_1)-h(x_2)\rangle\leq \|x_1-x_2\|\|h(x_1)-h(x_2)\|\leq \mu\|x_1-x_2\|^2.$
    This implies that $\langle x_1-x_2, q(x_1)-q(x_2)\rangle\geq (1-\mu)\|x_1-x_2\|^2$. The function $q(x)$ is trivially Lipschitz with constant $(1+\mu)$.

    \textit{b). } For $x_1,x_2\in\RR^d$, let $s=x-y$ and $t=h(x)-h(y)$. Then, $q(x)-q(y)=s-t$. We first note that $2\langle s-t,s\rangle = \|s-t\|^2+\|s\|^2-\|t\|^2$. Since $h(\cdot)$ is non-expansive, $\|t\|\leq \|s\|$. This implies that $2\langle s-t,s\rangle\geq \|s-t\|^2$, and hence $$\langle q(x)-q(y), x-y\rangle\geq 0.5\|q(x)-q(y)\|^2.$$
    This completes the proof for Lemma \ref{lemma:direction-1}.
\end{proof}

We now show the converse directions.
\begin{lemma}\label{lemma:direction-2}
Consider function $h:\RR^d\mapsto\RR^d$. Suppose $\cc,\ell>0$. Then, 
    \begin{enumerate}[label=\alph*).]
        \item If $h(\cdot)$ is $\cc$-strongly monotone and $\ell$-Lipschitz, then function $q(x)=x-(\cc/\ell^2)h(x)$ is $\sqrt{1-c^2/\ell^2}$-contractive. 
        \item If $h(\cdot)$ is $\cc$-co-coercive, then function $q(x)=x-2\cc h(x)$ is non-expansive.
    \end{enumerate}
\end{lemma}
\begin{proof}
    \textit{a).} For $x_1,x_2\in\RR^d$,
    \begin{align*}
        \|q(x_1)-q(x_2)\|^2&=\|x_1-x_2\|^2+\frac{\cc^2}{\ell^4}\|h(x_1)-h(x_2)\|^2-2\frac{\cc}{\ell^2}\langle h(x_1)-h(x_2),x_1-x_2\rangle\\
        &\leq \left(1+\frac{\cc^2}{\ell^4}\ell^2-2\frac{\cc}{\ell^2}\cc\right)\|x_1-x_2\|^2=(1-\cc^2/\ell^2)\|x_1-x_2\|^2.
    \end{align*}
    The inequality follows from the strongly monotone and Lipschitz nature of $h(\cdot)$.

    \textit{b).} For $x_1,x_2\in\RR^d$,
    \begin{align*}
        \|q(x_1)-q(x_2)\|^2&=\|x_1-x_2\|^2+4\cc^2\|h(x_1)-h(x_2)\|^2-4\cc\langle h(x_1)-h(x_2),x_1-x_2\rangle\\
        &\leq \|x_1-x_2\|^2+(4\cc^2-4\cc^2)\|h(x_1)-h(x_2)\|^2= \|x_1-x_2\|^2.
    \end{align*}
    The inequality follows from the co-coercive nature of $h(\cdot)$.
\end{proof}

\bibliographystyle{siamplain}
\bibliography{references}

\end{document}